\documentclass[12pt]{amsart}
\usepackage{amsfonts,amscd,amsthm,amsgen,amsmath,amssymb}
\usepackage[all]{xy}
\usepackage[vcentermath]{youngtab}
\usepackage[letterpaper,hmargin=1.3in,vmargin=1.1in]{geometry}
\newtheorem{theorem}{Theorem}[section]
\newtheorem{lemma}[theorem]{Lemma}
\newtheorem{corollary}[theorem]{Corollary}
\newtheorem{proposition}[theorem]{Proposition}

\theoremstyle{definition}

\newtheorem{example}[theorem]{Example}

\theoremstyle{remark}
\newtheorem{remark}[theorem]{Remark}

\numberwithin{equation}{section}

\begin{document}

\title{Moduli spaces of contact instantons}
\author{David Baraglia}
\author{Pedram Hekmati}

\address{School of Mathematical Sciences, The University of Adelaide, Adelaide SA 5005, Australia}

\email{david.baraglia@adelaide.edu.au}
\email{pedram.hekmati@adelaide.edu.au}

\begin{abstract}
We construct the moduli space of contact instantons, an analogue of Yang-Mills instantons defined for contact metric $5$-manifolds and initiate the study of their structure. In the $K$-contact case we give sufficient conditions for smoothness of the moduli space away from reducible connections and show the dimension is given by the index of an operator elliptic transverse to the Reeb foliation. The moduli spaces are shown to be K\"ahler when the $5$-manifold $M$ is Sasakian and hyperK\"ahler when $M$ is transverse Calabi-Yau. We show how the transverse index can be computed in various cases, in particular we compute the index for the toric Sasaki-Einstein spaces $Y^{p,q}$.
\end{abstract}
\thanks{This work is supported by the Australian Research Council Discovery Projects DP110103745, DP130102578 and DE12010265.}


\date{\today}



\maketitle


\section{Introduction}

The study of moduli spaces of the anti-self-dual instanton equation $*F_A = -F_A$ has generated stunning advances in our understanding of smooth $4$-manifolds. While in dimensions greater than $4$ the classification of smooth structures is far better understood, it is expected that higher dimensional instantons will prove a useful tool in the study of certain geometric structures on these manifolds. Higher dimensional instantons may be defined in $d \ge 4$ dimensions by choosing a $(d-4)$-form $\Omega$. We say that a connection $A$ is an {\em anti-self-dual $\Omega$-instanton} if the curvature $2$-form $F_A$ satisfies
\begin{equation}\label{equinstanton}
*F_A = -  \Omega \wedge F_A.
\end{equation}
Such equations were considered by physicists in \cite{cdfn} and further popularised by Donaldson and Thomas \cite{dt} and Tian \cite{tian}. Particular cases of (\ref{equinstanton}) include Hermitian-Einstein connections on K\"ahler manifolds, $G_2$- and $Spin(7)$-instantons \cite{dt}. In the special case where $\Omega$ is a closed form, as considered in \cite{tian}, we find on differentiating and using the Bianchi identity that a solution to (\ref{equinstanton}) is automatically a solution of the Yang-Mills equation $d_A ( *F_A) = 0$. In general $A$ is only a solution to the {\em Yang-Mills equation with torsion} \cite{hilp}
\begin{equation*}
d_A (*F_A) + d\Omega \wedge F_A = 0.
\end{equation*}
Somewhat surprisingly, there are special cases of the $\Omega$-instanton equation in which $\Omega$ is not closed, but for which every solution of (\ref{equinstanton}) is nevertheless a Yang-Mills connection because the term $d\Omega \wedge F_A$ automatically vanishes. This was observed for nearly K\"ahler manifolds in \cite{xu} and for geometries related to Killing spinors in \cite{hano}. A third such instance, the subject of this paper, is that of {\em contact instantons}, introduced by K\"all\'en and Zabzine in \cite{kaza} arising from the study of a $5$-dimensional super Yang-Mills theory. For this let $M$ be a contact metric $5$-manifold with contact form $\eta$. An {\em anti-self-dual contact instanton} is a solution to (\ref{equinstanton}) with $\Omega = \eta$, that is
\begin{equation*}
*F_A = -\eta \wedge F_A,
\end{equation*}
while a {\em self-dual contact instanton} is a solution with $\Omega = -\eta$,
\begin{equation*}
*F_A = \eta \wedge F_A.
\end{equation*}
Then even though $d\eta \neq 0$, the anti-self-dual contact instantons are Yang-Mills connections as one can show that $d\eta \wedge F_A = 0$.\\

In this paper we construct the moduli spaces of self-dual and anti-self-dual contact instantons on compact $K$-contact manifolds. Note that in order to construct a reasonable moduli space we find it necessary to assume the $K$-contact condition, that is, the Reeb vector field $\xi$ is a Killing vector for the metric on $M$. The Reeb vector field defines a $1$-dimensional foliation $\mathcal{F}_\xi$ on $M$ and we find that the transverse geometry of this foliation plays a significant role in understanding the contact instanton moduli spaces. For instance, we find that the dimension of the moduli space is given by the index of a complex which is transverse elliptic to the Reeb foliation.\\

The flow along $\xi$ defines a $1$-parameter group of isometries on $M$. As we recall in Section \ref{secci}, the closure of this $1$-parameter group in the isometry group of $M$ defines a torus $T^r$ of rank $r \ge 1$ acting on $M$ by isometries. When $r=1$, $M$ is {\em quasi-regular} and is a Seifert fibration over a symplectic $4$-orbifold $X$. In this case we prove (up to some minor details) that contact instantons on $M$ correspond to ordinary instantons on the orbifold $X$. On the other hand when $r > 1$, $M$ is {\em irregular} and the contact instanton equation can not be reduced to lower dimensions. We show that even when $M$ is irregular, it is possible to get a smooth moduli space of irreducible contact instantons under reasonable assumptions. We also prove that these moduli spaces are generally non-empty by looking at the case of the irregular $Y^{p,q}$ spaces, in Section \ref{secypq}.\\

Given a contact instanton $A$ on a principal $G$-bundle $P$ we prove in the irreducible case that the action of the torus $T^r$ lifts to the principal bundle and preserves $A$. From this we are able to re-express the dimension of the moduli space of contact instantons as the index of a complex transverse to a group action, as defined by Atiyah in \cite{at}. In general, computing the index of a complex elliptic transverse to a foliation is extremely difficult. That we are able to convert the index to one transverse to a group action is a notable simplification. While it is still a very difficult problem to obtain such an index, we isolate some cases in which it becomes possible to carry out the index computation to the end and determine the dimension of these moduli spaces.\\

When the contact manifold $M$ is Sasakian we can say more about the moduli spaces. The anti-self-duality equation can be interpreted as the Hermitian-Einstein equation for bundles with transverse holomorphic structure. We show in Section \ref{secsasaki} that the moduli space of irreducible anti-self-dual contact instantons on $M$ carries a natural K\"ahler structure. This extends a theorem of Biswas and Schumacher in which the moduli space was constructed in the quasi-regular case and shown to be K\"ahler \cite{bisc}. Furthermore, if $M$ has a transverse Calabi-Yau structure we prove that the moduli space in fact has a natural hyperK\"ahler structure.\\

The Reeb foliation $\mathcal{F}_\xi$ on $M$ has codimension $4$, so it is possible to speak of self-duality/anti-self-duality in the transverse directions. As explained in the paper, the contact instanton equation can be thought of as a transverse connection with self-dual/anti-self-dual curvature. One may generalise this to higher dimensional manifolds equipped with a codimension $4$ foliation and consider transverse anti-self-dual connections. Such instantons have recently been considered by Wang in \cite{wang}. It is therefore worth pointing out how our work differs from \cite{wang} and the features that are unique to the contact case. An important difference is that we consider moduli spaces of arbitrary solutions of (\ref{equinstanton}) and all possible deformations within the space of connections, while in \cite{wang} one fixes a transverse structure and considers only basic connections and basic deformations. It is a non-trivial result of our paper that under certain circumstances all sufficiently small deformations of the contact instanton equation are basic. Furthermore we prove a number of results specific to the contact case, such as the Gysin sequence of Proposition \ref{propcohom}, the K\"ahler structure in the Sasaki case and the lifting of the torus action $T^r$ in Proposition \ref{proptorus}.\\

We briefly outline the contents of the paper. In Section \ref{secci} we introduce the contact instanton equation along with a review of $K$-contact manifolds and transverse geometry. Section \ref{secmoduli} is concerned with the construction of the moduli space of contact instantons starting with the infinitesimal theory in \textsection \ref{secdeform} and the full deformation theory in \textsection \ref{secmodulispace}. In \textsection \ref{secvanish} we give conditions under which the moduli space of irreducible contact instantons is smooth. Section \ref{secsasaki} deals with the case where the contact manifold is Sasakian. In \textsection \ref{secgvanish} we give geometric conditions for smoothness of the moduli space in the Sasaki case and in \textsection \ref{secgmoduli} we prove that the smooth points of the moduli space have a K\"ahler structure. We finish with Section \ref{sectrind} in which we address the problem of calculating the transverse index and carry out the computation in special cases. In \textsection \ref{secqreg} we deal with the quasi-regular case and \textsection \ref{secypq} gives the computation for a family of irregular Sasaki-Einstein manifolds, the $Y^{p,q}$ spaces.


\section{Contact instantons}\label{secci}

Let $M$ be a manifold of dimension $2n+1$. Recall that an {\em almost contact metric structure} $(\xi,\eta,\Phi,g)$ on $M$ consists of a vector field $\xi$, $1$-form $\eta$, endomorphism $\Phi \colon TM \to TM$ and a Riemannian metric $g$ such that $\eta(\xi) =1$, $\Phi^2 = -I + \eta \otimes \xi$ and $g( \Phi X , \Phi Y) = g(X,Y) - \eta(X) \eta(Y)$ for all vector fields $X,Y$. Alternatively this is a reduction of structure of the tangent bundle to $U(n) \subset GL(2n+1,\mathbb{R})$. We let $V$ be the rank $1$ subbundle spanned by $\xi$ and $H = {\rm Ker}(\eta)$ the annihilator of $\eta$. Then we have an orthogonal decomposition $TM = V \oplus H$ together with a unitary structure on $H$. The restriction $J = \Phi|_H$ of $\Phi$ to $H$ defines the complex structure on $H$ and letting $\omega(X,Y) = g(X,\Phi Y)$, we find that $\omega$ is a $2$-form which restricted to $H$ is the Hermitian $2$-form associated to $J$. We say that $(\xi,\eta,\Phi,g)$ is a {\em contact metric structure} if in addition $d\eta = \omega$. This implies that $\eta$ is a contact form and $\xi$ the associated Reeb vector field. In this case we will also say that $M$ is a {\em contact metric manifold}.\\

In this paper we take $M$ to be a compact, connected $5$-manifold with contact metric structure $(\xi,\eta,\Phi,g)$. A differential form $\alpha \in \Omega^k(M)$ will be called {\em transverse} if $i_\xi \alpha = 0$. We let $\Omega^k_H(M) = \Gamma(M, \wedge^k H^*)$ denote the space of transverse $k$-forms. A transverse form $\alpha \in \Omega^k_H(M)$ is further said to be {\em basic} if $i_\xi d \alpha = 0$. If this is the case then $d\alpha$ is also basic. We let $\Omega^k_B(M)$ denote the space of basic $k$-forms and $d_B \colon \Omega^k_B(M) \to \Omega^{k+1}_B(M)$ the restriction of $d$ to basic forms. The cohomology of the complex $( \Omega^*_B(M) , d_B)$ will be called the {\em basic cohomology} of $M$ and denoted $H^*_B(M)$. Since $H$ is $4$-dimensional we have a decomposition $\Omega^2_H(M) = \Omega^+_H(M) \oplus \Omega^-_H(M)$ into self-dual and anti-self-dual transverse $2$-forms. A $2$-form $\alpha \in \Omega^2(M)$ is self-dual/anti-self-dual if and only if it satisfies $*\alpha = \pm \eta \wedge \alpha$, or equivalently $\alpha = \pm i_\xi (*\alpha)$. Following \cite{kaza} we may introduce the notion of self-dual/anti-self-dual connections on $M$.\\

Let $G$ be a compact Lie group with Lie algebra $\mathfrak{g}$ and let $P \to M$ be a principal $G$-bundle with connection $\nabla$. We say that $\nabla$ is a {\em self-dual contact instanton} (or SD contact instanton) if the curvature $F$ of $\nabla$ is self-dual
\begin{equation*}
*F = \eta \wedge F.
\end{equation*}
Similarly we say that $\nabla$ is an {\em anti-self-dual contact instanton} (or ASD contact instanton) if $F$ is anti-self-dual
\begin{equation*}
*F = -\eta \wedge F.
\end{equation*}

\begin{remark}
In the case of instantons in $4$ dimensions there is no essential difference between the self-dual and anti-self-dual cases, since these are interchanged by reversing the orientation of the $4$-manifold. In the case of contact instantons, a choice of orientation is distinguished by the contact structure. This leads to some important distinctions between the self-dual and anti-self-dual cases. As discussed in the introduction, we see that any ASD contact instanton satisfies the Yang-Mills equations $d_\nabla(*F) = 0$, while this is generally not the case for SD contact instantons.
\end{remark}

Recall that a contact metric structure $(\xi,\eta,\Phi,g)$ is said to be {\em $K$-contact} if $\xi$ is a Killing vector of $g$. In this case we will also say that $M$ is a {\em $K$-contact manifold}. If $M$ is $K$-contact, the Killing vector $\xi$ generates a $1$-parameter subgroup $\{ exp(t\xi) \}$ of $Isom(M,g)$, the isometry group of $(M,g)$. By the Myers-Steenrod theorem $Isom(M,g)$ is a compact Lie group acting smoothly on $M$. Let $T \subseteq Isom(M,g)$ be the closure of $\{ exp(t\xi) \}$. Since $T$ is a closed, connected, abelian subgroup of $Isom(M,g)$, it must be a torus of rank $r \ge 1$. That $M$ is a $K$-contact manifold imposes non-trivial restrictions on the rank of $T$. In fact, one can show that we have $1 \le r \le 3$ \cite{ruk}. We say that $M$ is {\em quasi-regular} when $r=1$ and {\em irregular} when $r > 1$.

\begin{example}
Let $(X,\omega)$ be a symplectic $4$-manifold such that the cohomology class of $\omega$ is integral and choose a lift of $[\omega] \in H^2(X,\mathbb{R})$ to $c \in H^2(X,\mathbb{Z})$. Let $\pi\colon M \to X$ be the principal circle bundle over $X$ with Chern class $c$. Then $M$ admits a $K$-contact structure $(\xi , \eta , \Phi , g)$ such that $\xi$ is the generator of the circle action and $\eta$ is a connection for the circle bundle $M \to X$. Such contact manifolds are called {\em regular} and the projection $\pi \colon  M \to X$ is known as the Boothby-Wang fibration.
If $\nabla$ is an instanton on $X$ then the pullback $\pi^* \nabla$ is a contact instanton on $M$. Not every contact instanton on $X$ arises in this manner. For instance $\nabla = d + i\eta$ on the trivial circle bundle over $X$ is a self-dual contact instanton which is not a pullback. Moreover $d + i\eta$ is not a Yang-Mills connection. For another example, let $X = \mathbb{CP}^2$ and for an integer $k > 1$ take as a symplectic form $k$ times the standard K\"ahler form on $\mathbb{CP}^2$. Then $M$ is the lens space $S^5/\mathbb{Z}_k$ and any non-trivial flat connection on $M$ is a contact instanton which is not a pullback from $X$.
\end{example}

\begin{example}\label{exorbi}
We can extend the construction of the previous example as follows. Suppose that $(X,\omega)$ is a symplectic $4$-orbifold such that the local uniformizing groups $\Gamma_x$ are cyclic for all $x \in X$ and suppose that $[\omega] \in H^2(X,\mathbb{R})$ admits a lift to a class $c \in H^2_{{\rm orb}}(X,\mathbb{Z})$, the degree $2$ orbifold cohomology of $X$. Then $c$ defines an orbifold principal circle bundle $\pi \colon  M \to X$. Recall that the structure of such a bundle involves homomorphisms $\varphi_x \colon  \Gamma_x \to U(1)$ of the local uniformizing groups and that $M$ is a manifold if and only if $\varphi_x$ is injective for each $x \in X$. In this case $M \to X$ is a Seifert fibration and $M$ admits the structure of a quasi-regular $K$-contact manifold. Now suppose that $P \to X$ is an orbifold principal $G$-bundle and that $\nabla$ is a connection on $P$ with self-dual/anti-self-dual curvature. Clearly $(P,\nabla)$ can be pulled back to define a contact instanton on $M$.
\end{example}

\begin{example}
The previous examples are effectively $4$-dimensional objects. In this example we show that there are contact instantons which can not be reduced to $4$ dimensions. For this take $M$ to be a Sasaki-Einstein $5$-manifold. In Section \ref{sectransverse} we will introduce the transverse Levi-Civita connection $\overline{\nabla}$ on $H$ and show in Proposition \ref{propeinst} that the induced connection on $\wedge^{\pm} H^*$ is self-dual/anti-self-dual. If $M$ is irregular this gives non-trivial examples of contact instantons on irregular contact manifolds. Note that compact irregular Sasaki-Einstein $5$-manifolds exist, for example we consider the $Y^{p,q}$ spaces of \cite{gmsw} in Section \ref{secypq}.
\end{example}


\subsection{Transverse bundles and connections}\label{sectransverse}

Let $\mathcal{F}$ be a foliation on a smooth manifold $M$ with tangent distribution $T\mathcal{F} \subseteq TM$ and let $\pi \colon  P \to M$ be a principal $G$-bundle. We say that $P$ is a {\em transverse} (or {\em foliated}) principal bundle \cite{mol} if there exists a foliation $\widetilde{\mathcal{F}}$ on $P$ with tangent distribution $T\widetilde{\mathcal{F}} \subseteq TP$ such that $\widetilde{\mathcal{F}}$ is $G$-invariant, $dim(\mathcal{F}) = dim(\widetilde{\mathcal{F}})$ and $\pi_*( T\widetilde{\mathcal{F}} ) = T\mathcal{F}$. An isomorphism of transverse principal $G$-bundles $(P,\widetilde{\mathcal{F}}),(P',\widetilde{\mathcal{F}}')$ is a principal bundle isomorphism $\phi \colon  P \to P'$ such that $\phi(\widetilde{\mathcal{F}}) = \widetilde{\mathcal{F}}'$. A connection $A \in \Omega^1(P,\mathfrak{g})$ on $P$ is called {\em transverse} (or {\em basic}) \cite{mol} if $A$ is basic with respect to the foliation $\widetilde{\mathcal{F}}$, that is $i_\xi A = \mathcal{L}_\xi (A) = 0$ for every vector field $\xi$ tangent to $\widetilde{\mathcal{F}}$. Every transverse structure on $P$ admits a transverse connection.\\

Let $\pi \colon  P \to M$ be a principal $G$-bundle on $M$ and $A$ a connection on $P$ with curvature $F_A$. If $F_A |_{T\mathcal{F}} = 0$ then the horizontal lift of $T\mathcal{F}$ with respect to $A$ is integrable and gives $P$ a transverse structure. With respect to this transverse structure we see that $A$ is transverse if and only if $i_\xi F_A = 0$ for all vector fields $\xi$ tangent to $\mathcal{F}$. Suppose this is the case. The space of transverse connections on $P$ with respect to this transverse structure is an affine space modelled on
\begin{equation*}
\Omega^1_B(M,\mathfrak{g}_P) = \{ \psi \in \Omega^1(M,\mathfrak{g}_P) \, | \, i_\xi \psi = 0, \, i_\xi d_A \psi = 0\, \,{\rm for} \, \, {\rm all} \, \xi \in \Gamma(M,T\mathcal{F}) \}.
\end{equation*}
Given a transverse structure $\widetilde{\mathcal{F}}$ on a principal $G$-bundle $P$ we may define basic characteristic classes as follows. Let $\nabla$ be a transverse connection on $P$ with curvature $F \in \Omega^2_B(M,\mathfrak{g}_P)$ and let $\varphi \in S^k(\mathfrak{g}^*)$ be an invariant polynomial on $\mathfrak{g}^*$. It is clear that $\varphi(F)$ is a closed basic $2k$-form on $M$ and thus defines a basic cohomology class $[ \varphi(F) ] \in H^{2k}_B(M)$. By a straightforward extension of the usual argument in Chern-Weil theory, we see that the cohomology class $[ \varphi(F) ]$ is independent of the choice of transverse connection. This shows that to every characteristic class $P \mapsto c(P) \in H^{ev}(M,\mathbb{R})$ defined over $\mathbb{R}$ there is a corresponding basic characteristic class $(P,\widetilde{\mathcal{F}}) \mapsto c_B(P,\widetilde{\mathcal{F}}) \in H^{ev}_B(M)$ such that $c_B(P,\widetilde{\mathcal{F}})$ is sent to $c(P)$ under the natural map $H^{ev}_B(M) \to H^{ev}(M,\mathbb{R})$.

\begin{example}
Let $M$ be a compact contact metric manifold of dimension at least $5$ and let $L = M \times \mathbb{C}$ be the trivial line bundle with $U(1)$-connection $\nabla = d + i\lambda \eta$, for $\lambda \in \mathbb{R}$. Then $\nabla$ has curvature $F = i\lambda \omega$. Let $E = L \oplus L^*$ be the associated $SU(2)$-bundle. The basic Pontryagin class $p_{1,B}(E)$ of $E$ is represented by
\begin{equation*}
-\frac{1}{8 \pi^2} Tr \left( \left[ \begin{matrix} i\lambda \omega & 0 \\ 0 & -i\lambda \omega \end{matrix} \right]^2 \right) = \frac{\lambda^2}{4 \pi^2} \omega^2.
\end{equation*}
Note that the class $[\omega^2]$ is non-trivial in $H^4_B(M)$ but maps to a trivial class in $H^4(M,\mathbb{R})$, since $\omega^2 = d( \eta \wedge \omega)$. This example shows that that basic characteristic classes can take on a continuous range of values and can be non-trivial even when the underlying bundle is trivial.
\end{example}

When $G = U(1)$ we may speak of transverse line bundles. With the aid of transverse connections we find that the group (under tensor product) of isomorphism classes of transverse $U(1)$-line bundles is given by the fibre product $H_B^2(M) \times_{H^2(M,\mathbb{R})} H^2(M,\mathbb{Z})$.\\

Let $\nabla$ be a contact instanton on a principal $G$-bundle $P \to M$ and let $F$ be the curvature. The self-dual/anti-self-dual condition on $F$ implies that $i_\xi F = 0$, where $\xi$ is the Reeb vector field. Thus, letting $\mathcal{F}$ be the foliation generated by $\xi$ we have that $P$ inherits the structure of a transverse principal bundle and $\nabla$ is a transverse connection. In understanding the moduli space of contact instantons it will be important to take into consideration the transverse structure of $P$.\\

On a compact $K$-contact manifold $M$ we have a long exact sequence relating basic cohomology to the usual cohomology of $M$ \cite{ton}:
\begin{equation}\label{gysin}
\cdots \to H_B^k(M) \buildrel j \over \longrightarrow H^k(M,\mathbb{R}) \to H^{k-1}_B(M) \buildrel \! \! \omega \wedge \over \longrightarrow H_B^{k+1}(M) \to \cdots
\end{equation}
called the {\em Gysin sequence} of $M$. In this sequence $j \colon  H_B^k(M) \to H^k(M,\mathbb{R})$ is the map induced by the inclusion $\Omega^k_B(M) \to \Omega^k(M)$ of basic forms. 
We let $d_V \colon  \Omega^k_H(M) \to \Omega^k_H(M)$ be given by $d_V \alpha = i_\xi d\alpha$ and $d_T \colon \Omega^k_H(M) \to \Omega^{k+1}_H(M)$ by $d_T \alpha = d\alpha - \eta \wedge d_V \alpha$. Let $d_T^* \colon  \Omega^k_H(M) \to \Omega^{k-1}_H(M)$ be the formal adjoint of $d_T$. If $M$ is a compact $K$-contact manifold then $d_T^*$ sends basic forms to basic forms. The restriction of $d_T^*$ to $\Omega^*_B(M)$ defines an operator $d_B^* \colon  \Omega^k_B(M) \to \Omega^{k-1}_B(M)$ which we may consider to be a formal adjoint to $d_B$. The {\em basic Laplacian} is then defined as $\Delta_B = d_B d_B^* + d_B^* d_B$. Then it is clear that a basic form $\alpha$ is basic harmonic, i.e. $\Delta_B \alpha = 0$, if and only if $d_B \alpha = d_B^* \alpha = 0$. Let $\mathcal{H}_B^k$ denote the space of basic harmonic $k$-forms. We have natural maps $\mathcal{H}_B^k \to H_B^k(M)$. From basic Hodge theory \cite{kato} these maps are isomorphisms when $M$ is $K$-contact. \\

Recall that when $M$ is $5$-dimensional we have a decomposition $\Omega^2_H(M) = \Omega^+_H(M) \oplus \Omega^-_H(M)$. If $M$ is $K$-contact then the Lie derivative $\mathcal{L}_\xi$ commutes with the Hodge star $*$ and we may speak of self-dual/anti-self-dual basic $2$-forms. This gives a corresponding decomposition $\Omega^2_B(M) = \Omega^+_B(M) \oplus \Omega^-_B(M)$ of basic $2$-forms. We now define groups $\mathcal{H}_B^{\pm}$ of basic harmonic $2$-forms which are self-dual/anti-self-dual giving a decomposition $\mathcal{H}_B^2 = \mathcal{H}_B^+ \oplus \mathcal{H}_B^-$. We let $H_B^{\pm}(M)$ denote the image of $\mathcal{H}_B^{\pm}$ in $H_B^2(M)$ and note that this gives isomorphisms $\mathcal{H}_B^{\pm} \simeq H_B^{\pm}(M)$. Since $\omega$ is self-dual, we may deduce from the Gysin sequence (\ref{gysin}) that the natural map $j \colon  \mathcal{H}_B^- \to H^2(M,\mathbb{R})$ is injective, while the natural map $j \colon  \mathcal{H}_B^+ \to H^2(M,\mathbb{R})$ has $1$-dimensional kernel spanned by $\omega$.

\begin{proposition}\label{propu1}
Let $M$ be a compact $K$-contact $5$-manifold and let $\mathcal{M}_{U(1)}^{\pm}$ denote the group (under tensor product) of isomorphism classes of $U(1)$-contact instantons on $M$. Then as abelian groups we have isomorphisms:
\begin{equation*}
\mathcal{M}_{U(1)}^{\pm} = \left( \mathcal{H}_B^{\pm}(M) \times_{H^2(M,\mathbb{R})} H^2(M,\mathbb{Z}) \right) \times \left( H^1(M,\mathbb{R})/H^1(M,\mathbb{Z}) \right).
\end{equation*}
\end{proposition}
\begin{proof}
We give the proof in the SD case, the ASD case being simpler. Let $\mathcal{C}^+$ be the space of basic closed self-dual $2$-forms which have integral periods. Let $j \colon H_B^+(M) \to H^2(M,\mathbb{R})$ be the natural map from basic to ordinary cohomology and $i \colon H^2(M,\mathbb{Z}) \to H^2(M,\mathbb{R})$ the map induced by $\mathbb{Z} \to \mathbb{R}$. Set $A = j(H_B^+(M)) \cap i(H^2(M,\mathbb{Z}))$ and note that $A$ is a finitely generated free abelian group. Using the identification $H_B^+(M) = \mathcal{H}_B^+$ we obtain a short exact sequence $0 \to Ker(j) \to \mathcal{C}^+ \to A \to 0$, which may be split giving $\mathcal{C}^+ = A \oplus Ker(j)$. We recall also that $Ker(j)$ is the $1$-dimensional space spanned by $\omega$. Let $\nabla$ be a $U(1)$-contact instanton and $F_\nabla$ the curvature. The map $\nabla \mapsto \frac{i}{2\pi}F_\nabla$ defines a homomorphism $f \colon \mathcal{M}_{U(1)}^+ \to \mathcal{C}^+$ which is clearly surjective. We claim that there exists a splitting $\mathcal{C}^+ \to \mathcal{M}_{U(1)}^+$. Since $A$ is a free abelian group it suffices to give a lift $Ker(j) \to \mathcal{M}_{U(1)}^+$. For $\lambda \omega \in Ker(j)$ we take the trivial line bundle with connection $d - 2\pi i\lambda \eta$, which gives the desired lift. Thus $\mathcal{M}_{U(1)}^+ \simeq \mathcal{C}^+ \oplus Ker(f)$. The kernel of $f$ is the space of isomorphism classes of flat connections on $M$:
\begin{equation*}
Ker(f) = H^2_{{\rm tors}}(M,\mathbb{Z}) \times H^1(M,\mathbb{R})/H^1(M,\mathbb{Z}),
\end{equation*}
where $H^2_{{\rm tors}}(M,\mathbb{Z})$ is the torsion subgroup of $H^2(M,\mathbb{Z})$. The proposition now follows by noting that $\mathcal{C}^+ \times H^2_{{\rm tors}}(M,\mathbb{Z}) = \mathcal{H}^+_B(M) \times_{H^2(M,\mathbb{R})} H^2(M,\mathbb{Z})$.
\end{proof}

\begin{remark}
From Proposition \ref{propu1} we see that every connected component of the moduli space of $U(1)$ ASD-contact instantons is a torus $T^{b^1(M)}$ of dimension $b^1(M)$ and that the transverse structure of the underlying line bundle is fixed on each component. In contrast the connected components in the SD case are products $\mathbb{R} \times T^{b^1(M)}$ of a torus and a real line and the transverse structure of the underlying line bundle changes as one moves in the $\mathbb{R}$-direction.
\end{remark}

Let $G$ be a compact Lie group and denote the centre of $G$ by $Z(G)$. If $P \to M$ is principal $G$-bundle and $\nabla$ a connection on $P$, we let $Aut(\nabla)$ denote the group of gauge transformations of $P$ which are $\nabla$-constant. By choosing a basepoint one can identify $Aut(\nabla)$ with a closed subgroup of $G$. Clearly $Aut(\nabla)$ must contain $Z(G)$. We say that $\nabla$ is {\em irreducible} if $Aut(\nabla)  = Z(G)$. If this is not the case we say $\nabla$ is {\em reducible}.

\begin{proposition}\label{proptorus}
Let $G$ be a compact, connected, semisimple Lie group with trivial centre and let $\nabla$ be an irreducible contact instanton on $P$. The torus action of $T$ on $M$ lifts to an action of $T$ on $P$ by principal bundle isomorphisms preserving $\nabla$.
\end{proposition}
\begin{proof}
Let $B$ be an invariant metric on the Lie algebra $\mathfrak{g}$ of $G$. Using the connection $\nabla$ to decompose $TP$ into horizontal and vertical subbundles, we obtain a $G$-invariant metric $g_P$ by using $B$ on the vertical bundle and $g$ on the horizontal. Let $\tilde{\phi}_t$ be the $1$-parameter family of diffeomorphisms of $M$ integrating $\xi$. By integrating the horizontal lift $\tilde{\xi}$ of $\xi$ we obtain a $1$-parameter family $\tilde{\phi}_t$ of principal bundle isomorphisms covering the $1$-parameter family $\phi_t$. Moreover if $A \in \Omega^1(P,\mathfrak{g})$ is the connection form for $\nabla$ then $\mathcal{L}_{\tilde{\xi}}A = i_\xi F = 0$, so the $1$-parameter family $\tilde{\phi}_t$ preserves $\nabla$. Observe that for each $t$, $\tilde{\phi}_t$ is an isometry of $(P,g_P)$. 

Recall that the torus $T$ is defined as the closure of $\{ \phi_t \}$ in $Isom(M,g)$. Similarly we define a torus $\tilde{T}$ as the closure of $\{ \tilde{\phi}_t \}$ in $Isom(P,g_P)$, noting that $P$ and hence $Isom(P,g_P)$ are compact. Since each $\tilde{\phi}_t$ is a principal bundle isomorphism preserving $\nabla$, the same is true of each $t \in \tilde{T}$. This defines a homomorphism $f \colon \tilde{T} \to Isom(M,g)$ sending $\tilde{\phi}_t$ to $\phi_t$. From this it follows that $f(\tilde{T}) = T$. 
Let $K$ be the kernel of $f \colon  \tilde{T} \to T$ and let $\psi \in K$. Then $\psi \colon  P \to P$ is a principal bundle isomorphism covering the identity on $M$ and preserving $\nabla$. So $\psi$ is a gauge transformation covariantly constant with respect to $\nabla$. Now as we assume $\nabla$ is irreducible and $G$ has trivial centre, $\psi$ must be the identity and $f \colon  \tilde{T} \to T$ is an isomorphism. This gives the desired lift of $T$.

\end{proof}

\begin{remark}
If $\nabla$ is reducible or $G$ has non-trivial centre then we do not necessarily obtain a lift of $T$ to automorphisms of $(P,\nabla)$, but as in the above proof we obtain a torus $\tilde{T}$ acting on $P$ by automorphisms and a surjection $f \colon  \tilde{T} \to T$ with the action of $\tilde{T}$ on $P$ covering the action of $T$ on $M$.
\end{remark}

\begin{remark}
From the lifted action of $T$ one can recover the transverse structure on $P$. One simply takes the vector field $\tilde{\xi}$ tangent to the action of the $1$-parameter subgroup $\{\phi_t \} \subseteq T$ and this defines the foliation $\tilde{\mathcal{F}}$ on $P$. More generally, a lift of the torus $T$ to an action on $P$ by principal bundle isomorphisms determines a transverse structure on $P$ in the same manner.
\end{remark}

Let $M$ be a $K$-contact $5$-manifold with Levi-Civita connection $\nabla$. In a local foliated coordinate chart the metric $g$ depends only on the transverse coordinates. It follows that the restriction of $g$ to the contact distribution $H$ corresponds locally to a metric on the space of leaves of the Reeb foliation. Thus it makes sense to speak of the transverse Levi-Civita connection of $M$ which is a metric connection $\overline{\nabla}$ on $H$. Alternatively we can define $\overline{\nabla}$ by the relation
\begin{equation}\label{equlevi}
\nabla_X Y = \overline{\nabla}_X Y + \Pi(X,Y) \xi,
\end{equation}
for all $X,Y \in \Gamma(M,H)$, where $\Pi$ is a section of $H^* \otimes H^*$. In fact from (\ref{equlevi}) we must have $\Pi(X,Y) = -g( Y , \nabla_X \xi )$. Since $\overline{\nabla}$ is a transverse connection its curvature $R_T$ is a section of $\wedge^2 H^* \otimes \wedge^2 H^*$. We call $R_T$ the transverse Riemannian curvature of $M$. Likewise we can define the transverse Ricci curvature $Ric_T$ and transverse scalar curvature $s_T$ of $M$. With our conventions we find
\begin{align*}
Ric_T(X,Y) &= Ric(X,Y) + \frac{1}{2}g(X,Y) \\
s_T &= s + 1
\end{align*} 
where $X,Y$ are horizontal. Following \cite{ahs} the transverse curvature $R_T$ can be viewed as a self-adjoint map $R_T\colon  \wedge^2 H^* \to \wedge^2 H^*$ which in an orthonormal frame $e_1, \dots , e_4$ takes the form $e^i \wedge e^j \mapsto \frac{1}{2} {R_T}_{ijkl} e^k \wedge e^l$. Under the decomposition $\wedge^2 H^* = \wedge^+ H^* \oplus \wedge^{-} H^*$ we have
\begin{equation}\label{equcurvdecom}
R_T = \left[ \begin{matrix} \frac{s_T}{12} + W^+_T & B_T \\ B_T^* & \frac{s_T}{12} + W^-_T \end{matrix} \right]
\end{equation}
where $W^\pm_T$ are the self-dual/anti-self-dual components of the transverse Weyl curvature and $B_T \colon  \wedge^- H^* \to \wedge^+ H^*$ corresponds to the trace-free part of the transverse Ricci curvature. We say that $M$ is {\em transverse Einstein} if the trace-free part of $Ric_T$ vanishes, in particular if $M$ is Einstein then it is automatically transverse Einstein. The trace-free part of $Ric_T$ vanishes if and only if $B_T = 0$ and we have:
\begin{proposition}\label{propeinst}
Let $M$ be a $K$-contact $5$-manifold which is transverse Einstein. The transverse Levi-Civita connection on $\wedge^+ H^*$ (resp. $\wedge^- H^*$) is a self-dual contact instanton (resp. anti-self-dual contact instanton) with structure group $SO(3)$. In either case the structure group lifts to $SU(2)$ if and only if $M$ is spin.
\end{proposition}


\section{Moduli spaces and deformations}\label{secmoduli}


\subsection{The deformation complex}\label{secdeform}

To study the local structure of the moduli space we consider the deformation theory of the contact instanton equations. We will focus mostly on the anti-self-dual case, adding remarks on the self-dual case when differences arise.\\

Suppose that $\nabla$ is an ASD contact instanton with curvature $F \in \Omega^{-}_{H}(M,\mathfrak{g}_P)$. Let $\Omega^k_H(M,\mathfrak{g}_P)$ be the space of sections $\alpha \in \Omega^k(M,\mathfrak{g}_P)$ for which $i_\xi \alpha = 0$. We let $d_V \colon  \Omega^k_H(M,\mathfrak{g}_P) \to \Omega^{k+1}_H(M,\mathfrak{g}_P)$ be given by $d_V \alpha = i_\xi d_\nabla \alpha$ and further define $d_T \colon  \Omega^k_H(M,\mathfrak{g}_P) \to \Omega^{k+1}_H(M,\mathfrak{g}_P)$ by $d_T \alpha = d_\nabla \alpha - \eta \wedge d_V \alpha$.

\begin{lemma}\label{lemideal}
Let $I \subset \Omega^*(M,\mathfrak{g}_P)$ denote the algebraic ideal of the graded Lie algebra $\Omega^*(M,\mathfrak{g}_P)$ generated by $\Omega^{-}_{H}(M,\mathfrak{g}_P)$. If $M$ is $K$-contact we have $d_\nabla I \subseteq I$.
\end{lemma}
\begin{proof}
It suffices to show that $d_\nabla ( \alpha \otimes \psi ) \in I$, where $\alpha \in \Omega^{-}_{H}(M)$ and $\psi \in \Omega^0(M,\mathfrak{g}_P)$. We have $d_\nabla( \alpha \otimes \psi) = \eta \wedge d_V \alpha \otimes \psi + d_H \alpha \otimes \psi + \alpha \wedge d_\nabla \psi$. The last two terms clearly belong to $I$ so it remains to show that $d_V \alpha = \mathcal{L}_\xi \alpha \in \Omega^{-}_{H}(M)$. However if $M$ is $K$-contact, $\xi$ is a Killing vector and it follows that $\mathcal{L}_\xi  \Omega^{-}_{H}(M) \subseteq \Omega^{-}_{H}(M)$ as required.
\end{proof}

We assume henceforth that $M$ is $K$-contact. Let $L^* = L^*(M,\mathfrak{g}_P)$ be the graded Lie algebra given by the quotient $\Omega^*(M,\mathfrak{g}_P)/I$. By Lemma \ref{lemideal} we have that $d_\nabla$ descends to a derivation $D \colon  L^k \to L^{k+1}$. Moreover since $(d_\nabla)^2 = F \in I$, we have $D^2 = 0$. Thus $(L^*, D)$ is a differential graded Lie algebra. Using the decomposition of forms induced by the splitting $TM = V \oplus H$, we obtain identifications
\begin{equation}\label{dgla}
\begin{aligned}
    L^0 &= \Omega^0(M,\mathfrak{g}_P), & L^1 &= \Omega^1(M,\mathfrak{g}_P), \\ 
    L^2 &= \Omega^{+}_{H}(M,\mathfrak{g}_P) \oplus \eta \wedge \Omega^1_H(M,\mathfrak{g}_P), & L^3 &= \eta \wedge \Omega^{+}_{H}(M,\mathfrak{g}_P).
\end{aligned}
\end{equation}
and $L^k = 0$ for $k > 3$. If $(L^* , D)$ is a differential graded Lie algebra recall that an element $\omega \in L^1$ is called a {\em Maurer-Cartan element} if it satisfies $D \omega + \frac{1}{2}[\omega , \omega] = 0$. From the definition of $L^*(M,\mathfrak{g}_P)$ we have:

\begin{proposition}
Let $\psi \in L^1(M,\mathfrak{g}_P) = \Omega^1(M,\mathfrak{g}_P)$. The connection $\nabla + \psi$ is an ASD contact instanton if and only if $\psi$ is a Maurer-Cartan element of $(L^*(M,\mathfrak{g}_P),D)$.
\end{proposition}

The terms of $L^*$ may be arranged into a complex, the {\em deformation complex} for $\nabla$:
\begin{equation}\label{equellcpx}
0 \longrightarrow L^0 \buildrel D \over \longrightarrow L^1 \buildrel D \over \longrightarrow L^2 \buildrel D \over \longrightarrow L^3 \longrightarrow 0.
\end{equation}

A direct computation shows that (\ref{equellcpx}) is an elliptic complex. If $M$ is compact then the associated cohomology groups, denoted $H^k(\mathfrak{g}_P)$, are finite dimensional. For $k=0,1,2$ we may interpret these groups in terms of the contact instanton equation as follows:
\begin{itemize}
\item{$H^0(\mathfrak{g}_P)$ is the Lie algebra of infinitesimal automorphisms of $\nabla$,}
\item{$H^1(\mathfrak{g}_P)$ represents infinitesimal deformations of $\nabla$ as a contact instanton,}
\item{$H^2(\mathfrak{g}_P)$ may be used to describe the obstruction to extending an infinitesimal deformation to a genuine deformation.}
\end{itemize}
It is of particular importance to calculate the dimension of $H^1(\mathfrak{g}_P)$, since this represents the expected dimension of the moduli space of contact instantons. As $M$ is odd-dimensional the index for the elliptic complex (\ref{equellcpx}) vanishes giving ${\rm dim}H^0(\mathfrak{g}_P) - {\rm dim}H^1(\mathfrak{g}_P) + {\rm dim}H^2(\mathfrak{g}_P) - {\rm dim}H^3(\mathfrak{g}_P) = 0$. This turns out not to be useful in determining the dimension of $H^1(\mathfrak{g}_P)$ and one has to work harder to determine the dimension of the moduli space.\\

The decomposition of $TM$ into vertical and horizontal components determines a bi-grading on $\Omega^*(M,\mathfrak{g}_P)$ by setting $\Omega^{p,q}(M,\mathfrak{g}_P) = \Gamma(M , \wedge^p H^* \otimes \wedge^q V^* \otimes \mathfrak{g}_P)$. The ideal $I$ generated by $\Omega^{-}_H(M,\mathfrak{g}_P)$ is a bi-graded ideal, hence the bi-grading passes to the quotient $L^*$, defining spaces $L^{p,q}$. Note that the contraction $i_\xi \colon \Omega^*(M,\mathfrak{g}_P) \to \Omega^{*-1}(M,\mathfrak{g}_P)$ sends $I$ to itself, so defines a contraction $i_\xi \colon  L^* \to L^{*-1}$ of bi-degree $(0,-1)$. Letting $L_H^k$ denote the kernel of $i_\xi$ on $L^k$, we find that $L^{k,0} = L_H^k$. Similarly the wedge operation $\eta \wedge \colon  \Omega^*(M,\mathfrak{g}_P) \to \Omega^{*+1}(M,\mathfrak{g}_P)$ descends to $\eta \wedge\colon  L^* \to L^{*+1}$ having bi-degree $(0,1)$. It is clear that $\eta \colon L^{k,0} \to L^{k,1}$ is an isomorphism and hence we identify $L^{k,1}$ with $L_H^k$. Similarly the wedge operation $\omega \wedge \colon  \Omega^*(M,\mathfrak{g}_P) \to \Omega^{*+2}(M,\mathfrak{g}_P)$ descends to an operator $L^* \to L^{*+2}$ which we continue to denote by $\alpha \mapsto \omega \wedge \alpha$.\\

Define $D_V \colon  L^*_H \to L^*_H$ of degree $0$ by $D_V \alpha = i_\xi D\alpha$ and $D_T\colon  L^*_H \to L^{*+1}_H$ of degree $1$ by $D_T \alpha = D \alpha - \eta \wedge D_V \alpha$. Then for $\alpha \in L^{k,0} = L^k_H$ we have $D\alpha = D_T \alpha + \eta \wedge D_V \alpha$, while for $\eta \wedge \beta \in L^{k,1} = \eta \wedge L^k_H$ we have $D(\eta \wedge \beta) = \omega \wedge \beta - \eta \wedge D_T \beta$. From this we see that the bi-complex $(L^{*,*} , D)$ has the following form:
\begin{equation*}
\xymatrix{
\Omega^0_H(M,\mathfrak{g}_P) \ar[r]^{-D_T} \ar@/^/[rrd]^(.7){\omega \wedge } & \Omega^1_H(M,\mathfrak{g}_P) \ar[r]^{-D_T} & \Omega^{+}_H(M,\mathfrak{g}_P) \\
\Omega^0_H(M,\mathfrak{g}_P) \ar[u]^{D_V} \ar[r]^{D_T} & \Omega^1_H(M,\mathfrak{g}_P) \ar[u]^{D_V} \ar[r]^{D_T} & \Omega^{+}_H(M,\mathfrak{g}_P) \ar[u]^{D_V}.
}
\end{equation*}
From $D^2 = 0$ we have $D_T D_V = D_V D_T$ and $D_T^2 \alpha = -\omega \wedge D_V \alpha$, where $\alpha \in \Omega^0(M,\mathfrak{g}_P)$.\\

The deformation complex for self-dual contact instantons is much the same, with one important distinction. Since $\omega$ is self-dual, terms of the form $\omega \wedge \psi$ are projected out and the complex takes the form
\begin{equation*}
\xymatrix{
\Omega^0_H(M,\mathfrak{g}_P) \ar[r]^{-D_T} & \Omega^1_H(M,\mathfrak{g}_P) \ar[r]^{-D_T} & \Omega^{-}_H(M,\mathfrak{g}_P) \\
\Omega^0_H(M,\mathfrak{g}_P) \ar[u]^{D_V} \ar[r]^{D_T} & \Omega^1_H(M,\mathfrak{g}_P) \ar[u]^{D_V} \ar[r]^{D_T} & \Omega^{-}_H(M,\mathfrak{g}_P) \ar[u]^{D_V}.
}
\end{equation*}
In this case we have $D_T D_V = D_V D_T$ and $D_T^2 =0$.\\

We say that an element $\alpha \in L^k$ is {\em basic} if $i_\xi \alpha = 0$ and $i_\xi D\alpha = 0$. This is precisely the kernel of $D_V \colon  L^k_H \to L^k_H$. We let $L_B^*$ denote the complex of basic forms. Observe that if $\alpha$ is basic then so is $D \alpha$, hence $D$ restricts to a differential $D_B \colon  L_B^k \to L_B^k$ on basic elements. This defines the {\em basic deformation complex}
\begin{equation}\label{basicdc}
0 \to \Omega^0_B(M,\mathfrak{g}_P)  \buildrel D_B \over \longrightarrow \Omega^1_B(M,\mathfrak{g}_P)  \buildrel D_B \over \longrightarrow \Omega^{+}_B(M,\mathfrak{g}_P) \to 0.
\end{equation}
We let $H^*_B(\mathfrak{g}_P)$ denote the cohomology of this complex. Note that $D_B$ is only defined on basic sections so this is not a complex of differential operators in the usual sense, much less an elliptic complex. The basic complex is however a {\em transverse elliptic complex}, transverse to the foliation of $M$ by the Reeb vector field $\xi$. This implies that the cohomology $H^*_B(\mathfrak{g}_P)$ is finite dimensional \cite{elk}.

\begin{proposition}\label{propcohom}
Suppose that $\nabla$ is an ASD contact instanton. We have a long exact sequence
\begin{equation}\label{exseqharm}
\dots \to H^{k-2}_B(\mathfrak{g}_P) \buildrel \! \! \omega \wedge \over \longrightarrow H^k_B(\mathfrak{g}_P) \to H^k(\mathfrak{g}_P) \to H^{k-1}_B(\mathfrak{g}_P) \buildrel \! \! \omega \wedge \over \longrightarrow H^{k+1}_B(\mathfrak{g}_P) \to \cdots.
\end{equation}
Suppose that $\nabla$ is an SD contact instanton. Then
\begin{equation*}
H^k(\mathfrak{g}_P) \simeq H^k_B(\mathfrak{g}_P) \oplus H^{k-1}_B(\mathfrak{g}_P).
\end{equation*}
\end{proposition}
\begin{proof}
We begin with the ASD case. Let $B( \, , \, )$ be an invariant metric on $\mathfrak{g}$ and use this to define an inner product $\langle \, , \, \rangle$ on $\Omega^*(M,\mathfrak{g}_P)$ by $\langle a,b \rangle = \int_M B( a , *b)$. Let $*_T \colon  \Omega^k_H(M) \to \Omega^{4-k}_H(M)$ denote the transverse Hodge star, which is related to the Hodge star on $M$ by $*_T b = (-1)^k i_\xi (*b)$, for $b \in \Omega^k_H(M)$. Therefore if $a,b \in \Omega^k_H(M,\mathfrak{g}_P)$, we have $\langle a , b \rangle = \int_M B(a , *_T b) \wedge \eta$. Let $D^*,D_V^*,D_T^*$ denote the formal adjoints of $D,D_V,D_T$ with respect to $\langle \, , \, \rangle$. We also let $\Lambda \colon  \Omega^{+}_H(M,\mathfrak{g}_P) \to \Omega^0_H(M,\mathfrak{g}_P)$ denote the adjoint of the operator $L = \omega \wedge  \colon  \Omega^0_H(M,\mathfrak{g}_P) \to \Omega^+_H(M,\mathfrak{g}_P)$. From the identity
\begin{equation*}
0 = \int_M \mathcal{L}_\xi B(a,*_T b) \wedge \eta = \int_M B( D_V a , *_T b) \wedge \eta + \int_M B(a , D_V *_T b) \wedge \eta 
\end{equation*}
we obtain $D_V^*(b) = (-1)^{k-1} *_T D_V( *_T b)$, where $b$ has degree $k$. Since $\xi$ is a Killing vector we see that $D_V$ and $*_T$ commute, giving $D_V^* = -D_V$. Taking the adjoint of the relation $D_T D_V = D_V D_T$ we obtain $D_V D_T^* = D_T^* D_V$.\\

Let $\Delta = DD^* + D^*D \colon  L^* \to L^*$ be the Laplacian associated to $D$. Further define $\Delta_T = D_T D_T^* + D_T^* D_T - D_V^2$. We have that $\Delta$ is elliptic and it follows that $\Delta_T$ is also elliptic since it has the same symbol as $\Delta$. We let $\mathcal{H}^k$ denote the kernel of $\Delta$ on $L^k$ and $\mathcal{H}_T^k$ the kernel of $\Delta_T$ on $L^k$. Since $\Delta_T$ respects the bi-grading we can further decompose the kernel of $\Delta_T$ into spaces $\mathcal{H}_T^{p,q}$. Taking the wedge product with $\eta$ gives an isomorphism $\mathcal{H}_T^{k,0} \simeq \mathcal{H}_T^{k,1}$. \\

From Hodge theory we have isomorphisms $H^k(\mathfrak{g}_P) \simeq \mathcal{H}^k$. We claim that similarly there are isomorphisms $H^k_B(\mathfrak{g}_P) \simeq \mathcal{H}_T^{k,0}$. Clearly an element $\alpha \in L^{k,0}$ is $\Delta_T$-harmonic if and only if $D_T \alpha = D_T^* \alpha = D_V \alpha = 0$. Thus $\alpha$ is basic and $D_B \alpha = 0$, so there is a natural map $f \colon \mathcal{H}_T^{k,0} \to H^k_B(\mathfrak{g}_P)$. Suppose that $f \alpha = 0$. Thus $\alpha = D_T \beta$, where $D_V \beta = 0$. Then $0 = D_T^* \alpha = D_T^* D_T \beta$, giving $\beta = 0$. Hence $f$ is injective. Now let $\alpha$ be closed and basic, that is $\alpha \in L^{k,0}$ with $D_T \alpha = D_V \alpha = 0$. Since $\Delta_T$ is elliptic, Hodge theory implies that there exists $\beta,\gamma \in L^{k,0}$ such that 
\begin{equation}\label{abg}
\alpha = \beta + \Delta_T \gamma
\end{equation}
with $\Delta_T \beta = 0$. Applying $D_V$ to (\ref{abg}) we obtain $\Delta_T D_V \gamma = 0$, hence in particular $D_V^2 \gamma = 0$ and thus $D_V \gamma = 0$. Applying $D_T$ to (\ref{abg}) we find $D_T D_T^* D_T \gamma = 0$. Taking the inner product with $D_T \gamma$, we find $D_T^* D_T \gamma = 0$ and thus (\ref{abg}) gives $\alpha = \beta + D_T D_T^* \gamma$. This shows that every cohomology class in $H_B^k(\mathfrak{g}_P)$ has a $\Delta_T$-harmonic representative, showing that $f$ is surjective. This proves the claim that $H_B^k(\mathfrak{g}_P) \simeq \mathcal{H}_T^{k,0}$.\\

Let $(A^*,d)$ be the complex with $A^k = \mathcal{H}^{k,0}_T \oplus \mathcal{H}^{k-1,0}_T$ and differential $d( \alpha , \beta ) = ( \omega \wedge \beta , 0)$. Note that this is a well-defined differential because either $\omega \wedge \beta = 0$ or $\beta \in \mathcal{H}^{0,0}_T$ in which case we clearly have $\omega \wedge \beta \in \mathcal{H}^{2,0}_T$. Let $j^* \colon (A^* , d) \to (L^* , D)$ be the chain map given by $j^*( \alpha , \beta ) = \alpha + \eta \wedge \beta$. To prove exactness of the sequence (\ref{exseqharm}) it will suffice to show that $j^*$ is a quasi-isomorphism. We must show that $j^k \colon H^k(A^*) \to H^k(L^*)$ is an isomorphism for $0 \le k \le 3$. The case $k=0$ is trivial. For the case $k=3$ it suffices to note that $\mathcal{H}^3 = \{ \alpha \in \Omega^+_H(M,\mathfrak{g}_P) \, | \, D_T^* \alpha = 0, \, D_V \alpha = 0 \} = \mathcal{H}_T^{2,0}$. The cases $k=1,2$ require more care. First consider a class $[\alpha] \in H^k(\mathfrak{g}_P)$, with $\alpha = \beta + \eta \wedge \gamma$ the $\Delta$-harmonic representative. Then $D \alpha = D^* \alpha = 0$ is equivalent to
\begin{equation*}
\begin{aligned}
D_V \beta - D_T \gamma &= 0 &D_T^* \beta - D_V \gamma &= 0\\
D_T \beta + \omega \gamma &= 0 & \Lambda \beta - D_T^* \gamma &= 0.
\end{aligned}
\end{equation*}
From this we find $D_T D_T^* \beta - D_V^2 \beta = D_T D_V \gamma - D_V D_T \gamma = 0$. Taking the inner product with $\beta$ we find $D_T^* \beta = D_V \beta = 0$. Thus also $D_T \gamma = 0$ and $D_V \gamma = 0$. If $k=2$, the $\omega \gamma$ term is zero, hence also $D_T \beta = 0$. This shows that every class $[\alpha] \in H^2(\mathfrak{g}_P)$ can be written as $\alpha = \beta + \eta \wedge \gamma$, where $\beta \in \mathcal{H}^{2,0}_T$, $\gamma \in \mathcal{H}^{1,0}_T$. Therefore $j^2 \colon H^2(A^*) \to H^2(L^*)$ is surjective. Now consider the case $k=1$. The term $\Lambda \beta$ vanishes, hence $D_T^* \gamma = 0$. Further, applying $D_T^*$ to $D_T \beta + \omega \gamma = 0$, we find
\begin{align*}
0 &= D_T^* D_T \beta + D_T^* (\omega \gamma) \\
&= D_T^* D_T \beta - *_T d_T  (\omega \gamma) \\
& = D_T^* D_T \beta - *_T  (\omega \wedge d_T \gamma )\\
& = D_T^* D_T \beta
\end{align*}
where we have used the fact that $D_T^* \colon \Omega^+_H(M,\mathfrak{g}_P) \to \Omega^1(M,\mathfrak{g}_P)$ is given by $-*_T d_T$. Taking the inner product with $\beta$ we have $D_T \beta = 0$, hence also $\omega \gamma = 0$. We have shown that $\beta,\gamma$ are $\Delta_T$-harmonic and that $\omega \gamma = 0$. This establishes that $j^1 \colon H^1(A^*) \to H^1(L^*)$ is an isomorphism.\\

It remains only to show that $j^2 \colon H^2(A^*) \to H^2(L^*)$ is an isomorphism. Since we have shown it is surjective it will suffice to show that $dim(H^2(A^*)) = dim(H^2(L^*))$. In fact we have $0 = \sum_k (-1)^k dim(H^k(L^*))$, since $(L^*,D)$ is an elliptic complex on a compact manifold of odd dimension. Similarly it is clear that $\sum_k (-1)^k dim(H^k(A^*)) \linebreak = 0$, hence $dim(H^2(A^*)) = dim(H^2(L^*))$ as claimed. This proves the proposition in the anti-self-dual case. The self-dual case is similar, but considerably easier since now we have an equality $\Delta = \Delta_T$. From this we have $\mathcal{H}^k = \mathcal{H}_T^{k,0} \oplus \mathcal{H}_T^{k-1,1}$ and thus $H^k(\mathfrak{g}_P) = H^k_B(\mathfrak{g}_P) \oplus H^{k-1}_B(\mathfrak{g}_P)$.
\end{proof}

\begin{proposition}
Let $\nabla$ be an ASD contact instanton. The map $\omega \colon  H^0_B(\mathfrak{g}_P) \to H^2_B(\mathfrak{g}_P)$ is injective.
\end{proposition}
\begin{proof}
Let $a \in H^0_B(\mathfrak{g}_P)$ be such that $[\omega \otimes a] = 0 \in H^2_B(\mathfrak{g}_P)$, so $a \otimes \omega = D_B b$ for some $b \in \Omega^1_B(M,\mathfrak{g}_P)$. Then
\begin{align*}
||a||^2 &= \frac{1}{2} \int_M B( a \omega , a \omega ) \wedge \eta \\
&= \frac{1}{2} \int_M B( D_T b , a \omega ) \wedge \eta \\
&=\frac{1}{2} \int_M d (  B( b , a \omega) \wedge \eta ) \\
&= 0.
\end{align*}
Thus $a = 0$, proving injectivity.
\end{proof}

\begin{corollary}\label{corbasic}
Let $\nabla$ be any ASD contact instanton or an irreducible SD contact instanton. We have an isomorphism $H^1_B(\mathfrak{g}_P) \simeq H^1(\mathfrak{g}_P)$ induced by the inclusion $L^k_B \to L^k$.
\end{corollary}


\subsection{The moduli space of contact instantons}\label{secmodulispace}

In this section we proceed to construct the moduli space of contact instantons and use an obstruction map to give a local description of this space. Our construction is modelled on the construction of the moduli space of instantons on a $4$-manifold as in \cite{ahs},\cite{dk}, which in turn are based on the Kuranishi approach to deformation theory. To simplify the presentation we will give the construction for ASD contact instantons. The SD case works identically.\\

Throughout we assume that $G$ is a compact, connected, semisimple Lie group with Lie algebra $\mathfrak{g}$. Fix a principal $G$-bundle $\pi \colon  P \to M$ and as usual let $\mathfrak{g}_P$ denote the adjoint bundle. To construct the moduli space of contact instantons on $P$ we introduce $L^2_k$-Sobolev norms, construct a moduli space of $L^2_k$-contact instantons for sufficiently large $k$ and argue that the moduli space so defined does not depend on the choice of $k$. To keep the notation simple we will hide the dependence on the underlying principal bundle $P$. Let $\mathcal{A}_k$ be the space of $L^2_k$-connections on $P$ and $\mathcal{G}_{k+1}$ the space of $L^2_{k+1}$-gauge transformations. We take $k$ large enough that Sobolev embedding holds. Then one shows as in \cite{dk} that for large enough $k$, $\mathcal{G}_{k+1}$ is a Hilbert Lie group acting smoothly on $\mathcal{A}_k$ and that the quotient $\mathcal{B}_k = \mathcal{A}_k / \mathcal{G}_{k+1}$ is Hausdorff in the quotient topology. We define the moduli space of ASD $L^2_k$-contact instantons to be the subspace $\mathcal{M}_k \subset \mathcal{B}_k$ of gauge equivalence classes of connections satisfying the contact instanton equation. By this definition, $\mathcal{M}_k$ is a Hausdorff topological space. We also let $\mathcal{A}_k^* \subseteq \mathcal{A}_k$ denote the subspace of irreducible $L^2_k$-connections and similarly define $\mathcal{B}_k^*$, $\mathcal{M}_k^*$.\\

Having defined $\mathcal{M}_k$ the next step is to give a local description of its topology. For this we turn to the standard deformation theory of the Maurer-Cartan equation. Let $\nabla$ be a contact instanton. Recall that $\nabla$ defines a deformation complex $(L^* , D)$, a differential graded Lie algebra and that contact instantons on $P$ correspond to Maurer-Cartan elements of $(L^*,D)$. From (\ref{dgla}) the spaces $L^m$ are smooth sections of vector bundles on $M$. We let $L^m_k$ denote the completion of $L^m$ in the $L^2_k$-Sobolev norm. Let $D^* \colon  L^m_k \to L^{m-1}_{k-1}$ be the formal adjoint of $D$, $\Delta = DD^* + D^*D \colon  L^m_k \to L^m_{k-2}$ the associated Laplacian, $H\colon  L^m_k \to L^m_{k}$ the projection to the $L^2$-orthogonal complement of $Ker(\Delta)$ and $G \colon  L^m_k \to L^m_{k+2}$ the Green's operator. We also set $\delta = D^*G$.\\

Let $F \colon  L^1_k \to L^1_k$ be the map $F(\alpha) = \alpha + \frac{1}{2} \delta [\alpha , \alpha]$. For large enough $k$, $F$ is a smooth map of a Hilbert space to itself. The derivative of $F$ at the origin is the identity, so in a neighbourhood of $0 \in L^1_k$ we have a smooth inverse map $F^{-1}$. Given $c > 0$ set $U_c = \{ \eta \in L^1_k \, | \, \Delta \eta = 0, \,  || \eta ||_k < c \}$. Then for small enough $c$ we have defined $F^{-1} \colon  U_c \to L^1_k$. Given $\eta \in U_c$, set $\alpha = F^{-1}(\eta)$. We claim that $\alpha$ is smooth. In fact, since $\eta = \alpha + \frac{1}{2} \delta [\alpha , \alpha]$, applying $\Delta$ gives $\Delta \alpha + \frac{1}{2} D^* [ \alpha , \alpha ] = 0$. Then $\alpha$ is smooth by elliptic regularity.

Let $\mathcal{H}^m(\mathfrak{g}_P)$ denote the space of $\Delta$-harmonic forms in $L^m$. Then $U_c$ is a neighbourhood of $0$ in $\mathcal{H}^1(\mathfrak{g}_P)$. Define $\Phi \colon  U_c \to \mathcal{H}^2(\mathfrak{g}_P)$ by setting $\Phi(\eta) = H [ F^{-1}(\eta) , F^{-1}(\eta)]$. We call $\Phi$ the {\em obstruction map} for the deformation complex $(L^*,D)$. Suppose that $\Phi(\eta) = 0$. Set $\alpha = F^{-1}(\eta)$, so that $\eta = \alpha + \frac{1}{2} \delta [\alpha , \alpha]$. Then $0 = D\eta = D\alpha + \frac{1}{2} D\delta [\alpha , \alpha]$. Next, we use the identity $D \delta = 1 - H - \delta D$ to obtain
\begin{equation}\label{almostmc}
D\alpha + \frac{1}{2}[\alpha , \alpha] - \delta [D\alpha , \alpha] = 0.
\end{equation}
We claim that $\delta [D\alpha , \alpha ] = 0$, provided $c$ is sufficiently small. In fact, from (\ref{almostmc}) we have
\begin{align*}
\delta[D\alpha , \alpha ] &= \delta[ -\frac{1}{2}[\alpha,\alpha] + \delta[D\alpha,\alpha] , \alpha] \\
&= \delta[ \delta[D\alpha , \alpha ] , \alpha ].
\end{align*}
Set $x = \delta[D\alpha , \alpha]$, so that $x = \delta[ x , \alpha]$. For large enough $k$ we obtain an estimate of the form $|| \delta[a,b] ||_k \le A ||a||_k ||b||_k$ for some constant $A > 0$. Hence we have $||x||_k \le A ||x||_k ||\alpha||_k$. For all sufficiently small $c$ we can assume $||\alpha ||_k < 1/A$, giving $x = \delta[D\alpha,\alpha] = 0$. Now (\ref{almostmc}) shows that $\alpha = F^{-1}(\eta)$ is a solution to the Maurer-Cartan equation. Moreover, $0 = D^* \eta = D^* \alpha + \frac{1}{2} D^* \delta [\alpha , \alpha] = D^* \alpha$, that is $D^* \alpha = 0$.\\

Let $Z = \{ \alpha \in L^1_k \, | \, D^*\alpha = 0, \, D\alpha + \frac{1}{2} [\alpha , \alpha] = 0 \}$. We have shown that $F^{-1}$ sends $\Phi^{-1}(0)$ into $Z$. Next we claim that all sufficiently small $\alpha \in Z$ are obtained this way. Given $\alpha \in Z$, set $\eta = F(\alpha) = \alpha + \frac{1}{2} \delta [\alpha , \alpha]$. Then clearly $D^* \eta = 0$. Also we find
\begin{align*}
D\eta &= D\alpha + \frac{1}{2} D\delta [\alpha , \alpha]\\
&= -\frac{1}{2}[\alpha , \alpha] + \frac{1}{2}( 1-H-\delta D)[\alpha , \alpha] \\
&= -\frac{1}{2} H[\alpha,\alpha] - \delta[ D\alpha , \alpha] \\
&= -\frac{1}{2} H[\alpha , \alpha].
\end{align*}
Thus $D\eta = -\frac{1}{2}H[\alpha,\alpha] = -\Phi(\eta)$. The left hand side is $D$-exact, while the right hand side is harmonic, hence we must have $D\eta = 0$, $\Phi(\eta) = 0$.\\

This shows that $F^{-1}$ sends $\Phi^{-1}(0)$ to a neighbourhood $W$ of $0 \in Z$. Let $Aut(\nabla)$ be the group of covariantly constant gauge transformations of $\nabla$ and set $\Gamma_\nabla = Aut(G)/ Z(G)$, where $Z(G)$ is the centre of $G$. In particular $\Gamma_\nabla = 1$ if and only if $\nabla$ is irreducible. As shown in \cite{dk}, there is a neighbourhood of $\nabla$ in $\mathcal{M}_k$ given by $W/\Gamma$. We have thus shown:
\begin{proposition}\label{proplocstruct}
For $c>0$, let $U_c = \{ \eta \in \mathcal{H}^1(\mathfrak{g}_P) | \, || \eta ||_k < c \}$. Choose $c$ sufficiently small so that the obstruction map $\Phi \colon U_c \to \mathcal{H}^2(\mathfrak{g}_P)$ is defined. Set $\Gamma_\nabla = Aut(G)/Z(G)$. For sufficiently small $c$ there is a neighbourhood of $\nabla$ in $\mathcal{M}_k$ given by $\Phi^{-1}(0)/\Gamma_\nabla$.
\end{proposition}

Next, let us address the issue of the dependence of the moduli space $\mathcal{M}_k$ on $k$. Clearly there is a natural map $\mathcal{M}_{k+1} \to \mathcal{M}_k$. Using the same argument as in \cite{dk}, we have:
\begin{proposition}
For all sufficiently large $k$, the natural map $\mathcal{M}_{k+1} \to \mathcal{M}_k$ is a homeomorphism.
\end{proposition}
We may now speak of the moduli space $\mathcal{M}$ of contact instantons and the open subspace $\mathcal{M}^* \subseteq \mathcal{M}$ of irreducible contact instantons.
From Proposition \ref{proplocstruct} we obtain:
\begin{corollary}
Let $\nabla \in \mathcal{M}^*$ be an irreducible contact instanton for which the obstruction map $\Phi$ vanishes. Then $\nabla$ has a neighborhood homeomorphic to a Euclidean space. If $\Phi$ vanishes for every $\nabla$ in $\mathcal{M}^*$, then $\mathcal{M}^*$ is a smooth manifold.
\end{corollary}
\begin{proof}
Only the last statement about smoothness requires explanation. As in \cite{ahs}, this follows from local universality of the spaces $F^{-1}(U_c)$ of solutions to the Maurer-Cartan equation, which serve as coordinate charts defining a smooth structure on $\mathcal{M}^*$.
\end{proof}

For a contact instanton $\nabla$, we define the transverse index $ind(\nabla)$ of $\nabla$ to be the index of the basic deformation complex (\ref{basicdc}), that is
\begin{equation*}
ind(\nabla) = dim( H^0_B(\mathfrak{g}_P)) - dim( H^1_B(\mathfrak{g}_P)) + dim( H^2_B(\mathfrak{g}_P)).
\end{equation*}
When $\nabla$ is irreducible and $H^2_B(\mathfrak{g}_P) = 0$, we have that $\mathcal{M}$ is smooth around $\nabla$ of dimension $dim(\mathcal{M}) = -ind(\nabla)$. The basic deformation complex is an example of a {\em transverse elliptic complex}.

\begin{remark}
Having constructed the moduli space of contact instantons, it is natural to attempt to compactify these spaces. In fact the problem of compactification has already been investigated by Wang in the case of transverse instantons \cite{wang}. These results can be applied to our moduli spaces opening up the exciting possibility of constructing Donaldson type invariants for contact $5$-manifolds.
\end{remark}


\subsection{Vanishing of obstructions}\label{secvanish}

We seek conditions under which the moduli space $\mathcal{M}^*$ of irreducible contact instantons is smooth. It is clear from Proposition \ref{propcohom} that we can not generally expect $H^2(\mathfrak{g}_P)$ to vanish, for if $\nabla$ is irreducible this would force $H^1(\mathfrak{g}_P)$ to also vanish, giving a $0$-dimensional moduli space. Fortunately we can prove vanishing of the obstruction map under the more reasonable condition that $H^2_B(\mathfrak{g}_P)$ vanishes:
\begin{proposition}\label{propunob}
Let $\nabla$ be an irreducible contact instanton (SD or ASD) for which $H^2_B(\mathfrak{g}_P) = 0$. Then the obstruction map $\Phi$ vanishes.
\end{proposition}
\begin{proof}
Consider first the ASD case. We will show that every infinitesimal deformation in $H^1(\mathfrak{g}_P)$ is tangent to a $1$-parameter family of deformations of $\nabla$. We then show that this forces $\Phi$ to vanish. Since $\nabla$ is irreducible Proposition \ref{propcohom} gives $H^1(\mathfrak{g}_P) \simeq H^1_B(\mathfrak{g}_P)$. As in the proof of Proposition \ref{propcohom} we set $\Delta_T = DD^* + D^*D - D_V^2$. Since $\Delta_T$ is elliptic, we have an $L^2$-decomposition into $Ker(\Delta_T)$ and $Ker(\Delta_T)^\perp$. Let $H_T$ be the projection to $Ker(\Delta_T)^\perp$. We define the Green's operator $G_T$ of $\Delta_T$ as the inverse of $\Delta_T$ on $Ker(\Delta_T)^\perp$. Set $\delta_T = D_T^* G_T$ and define $F_T \colon  (L^1_T)_k \to (L^1_T)_k$ by $F_T(\alpha) = \alpha + \frac{1}{2} \delta_T [\alpha , \alpha]$, where $(L^1_T)_k$ is the completion of $L^1_k$ in the $L^2_k$-Sobolev norm. For large enough $k$, $F_T$ is an isomorphism of Hilbert spaces. Let $c>0$ be small enough that $F_T^{-1}$ is defined on $U_c = \{ \eta \in \mathcal{H}^1_T \, | \, || \eta ||_k < c \}$. On $(L^2_T)_{k-1}$ we have $\Delta_T = D_T D_T^* - D_V^2$ and thus on $(L^1_T)_k$ we have
\begin{align*}
D_T \delta_T &= D_T D_T^* G_T \\
&= (\Delta_T + D_V^2) G_T \\
&= 1 - H_T + G_T D_V^2 \\
&= 1 + G_T D_V^2,
\end{align*}
where we have used the fact that $D_V$ and $G_T$ commute and that $H_T = 0$ on $(L^2_{k-1})_T$ since $\mathcal{H}^2_B(\mathfrak{g}_P) = 0$ by assumption. Now let $\eta \in U_c \subseteq \mathcal{H}_T^1$ and let $\alpha = F_T^{-1}(\eta)$, so that
\begin{equation}\label{etal}
\eta = \alpha + \frac{1}{2} \delta_T [\alpha , \alpha].
\end{equation}
Applying $D_V$ to (\ref{etal}) we have $0 = D_V \eta = D_V \alpha + \delta_T [D_V \alpha , \alpha]$. Arguing as in Section \ref{secmodulispace}, we have $\delta_T[D_V\alpha , \alpha] = 0$, provided $c$ is sufficiently small. Thus $D_V \alpha = 0$. Applying $D_T$ to (\ref{etal}) we find $0 = D_T \eta = D_\alpha + \frac{1}{2}[\alpha,\alpha]$, so $\alpha$ is a deformation of $\nabla$. Applying $D_T^*$ to (\ref{etal}) we have $0 = D_T^*\eta = D_T^* \alpha + \frac{1}{2}(D_T^*)^2 G_T [\alpha , \alpha]$. Now as $D_T^2 = -\omega \wedge D_V$, we have $(D_T^*)^2(\beta) = D_V( \Lambda \beta)$. Then it follows that $D_T^* \alpha = 0$.\\

Let $Z = \{ \alpha \in L^1_k \, | \, D^*\alpha = 0, \, D\alpha + \frac{1}{2} [\alpha , \alpha] = 0 \}$. In Section \ref{secmodulispace} we saw that $F$ sends a neighbourhood of $0 \in Z$ to a neighbourhood of $\Phi^{-1}(0)$ in $U_c$. On the other hand we have just seen that $F_T^{-1}$ defines a map $F_T^{-1} \colon  U_c \to Z$. This shows that $\Phi \circ F \circ F_T^{-1} = 0$. However the differential of $F \circ F_T^{-1} \colon  U_c \to U_c$ at $0$ is the identity. This shows that $\Phi = 0$ in a neighbourhood of $0$, or $\Phi = 0$ for all sufficiently small $c$. This completes the proof in the ASD case. The SD case is similar with the only difference being that $(D_T^*)^2 = 0$ in this case.
\end{proof}

\begin{corollary}\label{corfixed}
Let $\nabla$ be an irreducible contact instanton (SD or ASD) for which $H^2_B(\mathfrak{g}_P) = 0$. There is an open neighborhood of $\nabla$ in the moduli space $\mathcal{M}^*$ over which the transverse structure on $P$ remains fixed.
\end{corollary}
\begin{proof}
In the proof of Proposition \ref{propunob} we see that all nearby contact instantons are of the form $\nabla + \alpha$, where $\alpha \in \Omega^1_B(M,\mathfrak{g}_P)$. Therefore $\nabla$ and $\nabla + \alpha$ induce the same transverse structure on $P$.
\end{proof}


\section{Contact instantons on Sasaki $5$-manifolds}\label{secsasaki}

Recall that a contact metric structure $(\xi,\eta,\Phi,g)$ on $M$ is {\em Sasakian} if for all vector fields $X,Y$ on $M$ we have $(\nabla_X \Phi)Y = g(X,Y)\xi - \eta(Y)X$, where $\nabla$ is the Levi-Civita connection for $g$. In this case we say that $M$ is a {\em Sasaki manifold}. Note also that Sasaki manifolds are automatically $K$-contact. In this section we take $M$ to be a compact, connected Sasaki $5$-manifold and consider the moduli spaces of contact instantons on $M$.

\subsection{Vanishing theorems}\label{secgvanish}

In \cite{it} a vanishing theorem for instantons on compact K\"ahler $4$-manifolds is proven. An identical argument can be applied to the basic deformation complex (\ref{basicdc}) on a compact Sasaki $5$-manifold giving:

\begin{proposition}\label{propvanish}
Suppose that $M$ is a compact Sasaki $5$-manifold with positive transverse scalar curvature $s_T > 0$ and let $\nabla$ be an irreducible anti-self-dual contact instanton. Then $H^2_B(\mathfrak{g}_P) = 0$. If $s_T = 0$ then every element of $H^2_B(\mathfrak{g}_P)$ is covariantly constant. 
\end{proposition}
Recall that a Sasaki structure $(\xi,\eta,\Phi,g)$ is {\em Sasaki-Einstein} if $g$ is an Einstein metric. In this case the transverse scalar curvature is automatically positive, and we have:
\begin{corollary}
Let $M$ be a compact Sasaki-Einstein $5$-manifold and $P \to M$ a principal $G$-bundle. The moduli space $\mathcal{M}^*$ of irreducible ASD contact instantons on $P$ is smooth.
\end{corollary}

\begin{proposition}
Let $M$ be a compact $K$-contact $5$-manifold with positive transverse scalar curvature and with $W_T^- = 0$ (resp. $W_T^{+}=0$). Then $H^2_B(\mathfrak{g}_P) = 0$ for any irreducible self-dual (resp. anti-self-dual) contact instanton.
\end{proposition}

We say that a Sasaki $5$-manifold $M$ is {\em transverse Calabi-Yau} if $M$ has transverse complex structures $J_1 = J,J_2,J_3$ satisfying the quaternion relations $J_1 J_2 = J_3$, such that the $J_i$ are hermitian with respect to $g|_H$ and the associated transverse K\"ahler forms $\omega_1 = \omega, \omega_2, \omega_3$ are closed (in higher dimensions this would be the definition of a transverse hyperK\"ahler structure). The transverse Levi-Civita connection on $\wedge^+ H^*$ is flat, since $\omega_1,\omega_2,\omega_3$ are covariantly constant. By Equation (\ref{equcurvdecom}), we see that $W_T^+$ and $s_T$ both vanish.

\begin{corollary}\label{corcy}
Let $M$ be a compact Sasaki $5$-manifold with transverse Calabi-Yau structure and $P \to M$ a principal $G$-bundle. The moduli space $\mathcal{M}^*$ of irreducible ASD contact instantons on $P$ is smooth.
\end{corollary}

\begin{remark}\label{remtranscy}
Any compact Sasakian manifold with transverse Calabi-Yau structure is automatically quasi-regular \cite{bg}, so is an orbifold circle bundle over a Calabi-Yau orbifold.
\end{remark}


\subsection{Transverse holomorphic bundles}\label{secholo}

On a Sasaki $5$-manifold $M$ we have an identification $\Omega^+_H(M,\mathfrak{g}_P)_{\mathbb{C}} = \Omega^{2,0}_H(M,\mathfrak{g}_P) \oplus \left( \Omega^0_H(M,\mathfrak{g}_P)_{\mathbb{C}} \wedge \omega \right) \oplus \Omega^{0,2}_H(M,\mathfrak{g}_P)$. Thus a connection $A$ on $P$ with curvature $F$ is an ASD contact instanton if and only if
\begin{equation*}
\begin{aligned}
i_\xi F &= 0, & F^{2,0} &= 0, & \Lambda F &= 0.
\end{aligned}
\end{equation*}
As usual the condition $i_\xi F = 0$ gives $P$ a transverse structure and $A$ is a transverse connection. This means that $P$ admits local trivialisations over foliated charts $\{ U_\alpha \}$ for which the transition functions $g_{\alpha \beta} \colon  U_{\alpha \beta} \to G$ depend only on the transverse coordinates. The condition $F^{2,0} = 0$ defines an integrable $\overline{\partial}$-connection on $P$. Thus $P$ is a {\em transverse holomorphic bundle}. Since $G$ is compact it has a complexification $G_{\mathbb{C}}$. To say that $P$ has a transverse holomorphic structure means that the local trivialisations of $P$ can be chosen so that the transition functions $g_{\alpha \beta} \colon  U_{\alpha \beta} \to G_{\mathbb{C}}$ are holomorphic functions of the transverse coordinates. We interpret the condition $\Lambda F = 0$ as saying that $A$ is a {\em transverse Hermitian-Einstein connection}.\\

Let $A$ be an ASD contact instanton on $P$ and write $d_A = \partial_A + \overline{\partial}_A$. To indicate the dependence on $A$ we will use notation such as $D_{T,A}$, $H_{B,A}^*(\mathfrak{g}_P)$ and $\mathcal{H}^*_{T,A}$ for $D_T$, $H_B^*(\mathfrak{g}_P)$ and $\mathcal{H}^*_T$. Using the decomposition into $(1,0)$- and $(0,1)$-forms, the basic deformation complex (\ref{basicdc}) decomposes (over $\mathbb{C}$) as follows:
\begin{equation*}
\xymatrix{
& \Omega^{1,0}_B(M,\mathfrak{g}_P) \ar[r]^{\partial_A} \ar[dr]^{\Lambda \overline{\partial}_A} & \Omega^{2,0}_B(M,\mathfrak{g}_P) \\
\Omega^{0,0}_B(M,\mathfrak{g}_P) \ar[ur]^{\partial_A} \ar[dr]^{\overline{\partial}_A}& & \Omega^{0,0}_B(M,\mathfrak{g}_P) \\
& \Omega^{0,1}_B(M,\mathfrak{g}_P) \ar[r]^{\overline{\partial}_A} \ar[ur]^{\Lambda \partial_A} & \Omega^{0,2}_B(M,\mathfrak{g}_P)
}
\end{equation*}
Since $M$ is Sasakian and since the K\"ahler identities are local we have the {\em transverse K\"ahler identities}:
\begin{equation*}
\begin{aligned}
\partial_A^* &= -i[ \overline{\partial}_A , \Lambda ], & \overline{\partial}_A^* &= i[\partial_A , \Lambda].
\end{aligned}
\end{equation*}
Let $\Delta_{\overline{\partial}_A} = \overline{\partial}_A \overline{\partial}_A^* + \overline{\partial}_A^* \overline{\partial}_A \colon  \Omega^{p,q}_B(M,\mathfrak{g}_P) \to \Omega^{p,q}_B(M,\mathfrak{g}_P)$ be the Laplacian associated to $\overline{\partial}_A$ and let $\mathcal{H}^{p,q}_{\overline{\partial}_A}$ denote the space of basic $\overline{\partial}_A$-harmonic sections of $\Omega^{p,q}_B(M,\mathfrak{g}_P)$.
\begin{proposition}\label{prophodge}
We have isomorphisms
\begin{equation*}
\begin{aligned}
H^0_{B,A}(\mathfrak{g}_P)_{\mathbb{C}} & \simeq \mathcal{H}^{0,0}_{\overline{\partial}_A}, & H^1_{B,A}(\mathfrak{g}_P)_{\mathbb{C}} &\simeq \mathcal{H}^{0,1}_{\overline{\partial}_A} \oplus \overline{\mathcal{H}^{0,1}_{\overline{\partial}_A}}, & H^2_{B,A}(\mathfrak{g}_P)_{\mathbb{C}} &\simeq \mathcal{H}^{0,2}_{\overline{\partial}_A} \oplus \overline{\mathcal{H}^{0,2}_{\overline{\partial}_A}} \oplus \mathcal{H}^{0,0}_{\overline{\partial}_A}.
\end{aligned}
\end{equation*}
\end{proposition}
\begin{proof}
The proof is a straightforward application of the transverse K\"ahler identities. We sketch the details for the degree $1$ case. Let $\alpha \in \Omega^1_B(M,\mathfrak{g}_P)_{\mathbb{C}}$ and write $\alpha = a + \overline{b}$ for $a,b \in \Omega^{0,1}_B(M,\mathfrak{g}_P)$. Then $\alpha$ is $D_{T,A}$-harmonic if and only if
\begin{equation*}
\begin{aligned}
\overline{\partial}_A a &= 0, & \Lambda (\partial_A a + \overline{\partial_A b} ) & = 0, \\
\overline{\partial}_A b &= 0, & \overline{\partial}_A^* a  + \partial_A^* \overline{b} & = 0.
\end{aligned}
\end{equation*}
From the K\"ahler identites this is equivalent to
\begin{equation*}
\overline{\partial}_A a = \overline{\partial}_A b = \Lambda \partial_A a = \Lambda \partial_A b = 0.
\end{equation*}
Using the K\"ahler identities a second time this is equivalent to 
\begin{equation*}
\overline{\partial}_A a = \overline{\partial}_A b = \overline{\partial}_A^* a = \overline{\partial}_A^* b = 0
\end{equation*}
which is exactly the condition that $a,b$ are $\overline{\partial}_A$-harmonic.
\end{proof}


\subsection{Geometry of the moduli space}\label{secgmoduli}

We show that the moduli space $\mathcal{M}^*$ of ASD-contact instantons on a compact Sasaki $5$-manifold $M$ with positive transverse scalar curvature is a K\"ahler manifold. Throughout this section we either assume that $M$ has positive transverse scalar curvature so that the moduli space $\mathcal{M}^*$ is smooth, or we restrict to the smooth points of $\mathcal{M}^*$. As usual choose an invariant metric $B( \, , \, )$ on $\mathfrak{g}$. We let $h$ be the natural $L^2$-metric on $\mathcal{M}^*$, that is for a gauge equivalence class $[A] \in \mathcal{M}^*$ we have
\begin{equation}\label{metric}
h_A( \alpha , \beta ) = \int_M B( \alpha \wedge *_T \beta ) \wedge \eta
\end{equation}
where $\alpha,\beta$ are the harmonic representatives of classes in $H^1_{B,A}(\mathfrak{g}_P) \simeq T_A \mathcal{M}^*$. Note that in order to show that $h$ is a well-defined Riemannian metric on $\mathcal{M}^*$ we need to check that $h$ does not depend on the choice of connection $A$ representing the equivalence class $[A] \in \mathcal{M}^*$. Under a gauge transform $d+A \mapsto g^{-1}(d+A)g$, the infinitesimal deformations $\alpha,\beta$ map to $g^{-1}\alpha g$, $g^{-1} \beta g$, which are the corresponding harmonic representatives. By gauge invariance of (\ref{metric}) we have that $h$ does not depend on the choice of representative $A$.\\

Next we define an almost complex structure $\mathcal{J}$ on $\mathcal{M}^*$. By the proof of Proposition \ref{prophodge} we have that $\alpha \in \Omega^1_B(M,\mathfrak{g}_P)$ is $D_{T,A}$-harmonic if and only if $\alpha = a + \overline{a}$, where $a \in \Omega^{0,1}(M,\mathfrak{g}_P)$ is $\overline{\partial}_A$-harmonic. It follows that the space of $D_{T,A}$-harmonic $1$-forms is closed under the action of the complex structure $J$ and thus induces an almost complex structure $\mathcal{J}$ on $\mathcal{M}^*$. Let $\Omega$ be the bilinear form on $T\mathcal{M}^*$ given by $\Omega(\alpha,\beta) = h( \alpha , \mathcal{J}\beta)$. Using the identity $*_T \alpha = \frac{1}{2} J\alpha \wedge \omega$ for $\alpha \in \Omega^1_H(M,\mathfrak{g}_P)$, we find
\begin{equation}\label{equomega}
\Omega_A( \alpha , \beta ) = \frac{1}{2} \int_M B( \alpha \wedge \beta ) \wedge \omega \wedge \eta,
\end{equation}
where $\alpha,\beta$ are the harmonic forms representing elements of $H^1_{B,A}(\mathfrak{g}_P)$. This shows that $\Omega$ is skew-symmetric, so $h$ is hermitian with respect to $J$ and $\Omega$ is the associated $2$-form.

\begin{proposition}\label{propkahler}
The complex structure $\mathcal{J}$ is integrable and the hermitian form $\Omega$ is closed, hence $\mathcal{M}^*$ has a natural K\"ahler structure.
\end{proposition}
\begin{remark}
This proposition extends a result of Biswas and Schumacher, in which the moduli space is constructed for quasi-regular Sasaki manifolds and shown to be K\"ahler \cite{bisc}. 
\end{remark}
\begin{proof}
The proof is similar to \cite{it2}. The deformation theory of the contact instanton equations in Section \ref{secmodulispace} provides distinguished coordinate systems on $\mathcal{M}^*$. We show that these are normal coordinates for the metric $h$ and that in these coordinates the $1$-jet of $\Omega$ at the origin is constant. This will show that $\Omega$ is closed and $\mathcal{J}$ is integrable.\\

Let $A \in \mathcal{M}^*$ be an irreducible contact instanton. Recall from Section \ref{secmodulispace} that there is an open neighborhood $U$ of $0 \in H^1_{B,A}(\mathfrak{g}_P)$ over which we have a map $F^{-1}_T \colon U \to Z_A = \{ \alpha \in \Omega^1_B(M,\mathfrak{g}_P) \, | \, D_{T,A}^* \alpha = 0, \, D_{T,A}\alpha + \frac{1}{2}[\alpha , \alpha ] = 0 \}$ and that this gives local coordinates on $\mathcal{M}^*$ centred at $A$. Here $\Omega^k_B(M,\mathfrak{g}_P)$ denotes $\mathfrak{g}_P$-valued forms which are basic with respect to $A$, but since $F^{-1}_T(x)$ is itself basic with respect to $A$, we have that $A$ and $A + F^{-1}_T(x)$ define the same spaces of basic forms. For this reason the connections $A$ and $A + F^{-1}_T(x)$ define the same operator $D_V$. We will write $f\colon  U \to Z_A$ for $F^{-1}_T$. Thus a point $x \in U \subseteq H^1_{B,A}(\mathfrak{g}_P)$ corresponds to the connection $A+ f(x) \in \mathcal{M}^*$. Under the identification $H^1_{B,A}(\mathfrak{g}_P) = \mathcal{H}^1_{T,A}$ we can take $x$ to be a harmonic $1$-form. By definition of $f(x) = F_T^{-1}(x)$, we have $x = f(x) + \frac{1}{2} \delta_{T,A} [ f(x) , f(x) ]$, where $\delta_{T,A} = D_{T,A}^* G_{T,A}$. It follows that $f(x) = x - \frac{1}{2} \delta_{T,A} [ x , x ] + r_1(x)$, where $r_1(x)$ has vanishing $2$-jet at $x=0$. Let $\lambda \in \mathcal{H}^1_{T,A}$, then $\lambda$ defines a vector field $\partial_\lambda$ on $U$ and we have
\begin{equation*}
f_*(\partial_\lambda)(x) = \partial_\lambda f(x) = \lambda - \delta_{T,A} [ \lambda , x ] + r_2(x,\lambda),
\end{equation*}
where for fixed $\lambda$, $r_2(x,\lambda)$ has vanishing $1$-jet at $x=0$. Note that $\partial_\lambda f(x)$ is a locally defined vector field on $\mathcal{M}^*$. More specifically $\partial_\lambda f(x)$ represents a class in $H^1_{B,A+f(x)}(\mathfrak{g}_P)$ which is the deformation of $A + f(x)$ in the direction of the vector field. In general $\partial_\lambda f(x)$ is not the harmonic representative. Let us denote by $a(x,\lambda)$ the harmonic form representing $[ \partial_\lambda f(x)]$ so that there exists a $\mu(x,\lambda) \in \Omega^0_B(M,\mathfrak{g}_P)$ such that $a(x,\lambda) = \partial_\lambda f(x) + D_{T,A+f(x)} \mu(x,\lambda)$. For this to be harmonic we need $D_{T,A+f(x)}^* a(x,\lambda) = 0$, that is
\begin{equation}\label{equmu}
D_{T,A+f(x)}^* D_{T,A+f(x)} \mu(x,\lambda) + D_{T,A+f(x)}^* \partial_\lambda f(x) = 0.
\end{equation}
This is not quite an elliptic equation, but we can easily remedy this for if $\mu(x,\lambda)$ satisfies (\ref{equmu}), then since $D_V \mu(x,\lambda) = 0$ we have
\begin{equation}\label{equmu2}
D_{T,A+f(x)}^* D_{T,A+f(x)} \mu(x,\lambda) - D_V^2 \mu(x,\lambda) + D_{T,A+f(x)}^* \partial_\lambda f(x) = 0,
\end{equation}
which clearly is elliptic. For fixed $x,\lambda$, any two solutions $\mu_1,\mu_2 \in \Omega^0(M,\mathfrak{g}_P)$ of (\ref{equmu2}) differ by an element of $\mathcal{H}_{B,A+f(x)}^0$. But since $A + f(x)$ is irreducible this space is trivial, hence $\mu(x,\lambda)$ is the unique solution to (\ref{equmu2}). The upshot is that by standard elliptic theory we have that $\mu(x,\lambda)$ depends smoothly on $x$ (it is clearly linear in $\lambda$). Lastly note that $\mu(0,\lambda) = 0$, by uniqueness.\\

Now we can show that $x \mapsto f(x)$ gives normal coordinates centred at $A$. We use $ + \dots $ to denote terms whose $1$-jet vanishes at $x=0$. Let $\lambda_1,\lambda_2 \in \mathcal{H}^1_{T,A}$ define vector fields $\partial_{\lambda_1},\partial_{\lambda_2}$. We will determine the $1$-jet of $h( \partial_{\lambda_1} , \partial_{\lambda_2})$ at $x=0$. Note first that
\begin{align*}
a(x,\lambda) &= \lambda - \delta_{T,A} [\lambda , x] + D_{T,A+f(x)} \mu(x,\lambda) + \dots \\
&= \lambda - \delta_{T,A}[\lambda,x] + D_{T,A} \mu(x,\lambda) + \dots
\end{align*}
Letting $\langle \, , \, \rangle$ denote the $L^2$-inner product on $\Omega^*(M,\mathfrak{g}_P)$ we have:
\begin{align*}
h( \partial_{\lambda_1} , \partial_{\lambda_2} ) &= \langle a(x,\lambda_1) , a(x,\lambda_2) ) \rangle \\
&= \langle \lambda_1 - \delta_{T,A} [\lambda_1 , x ] + D_{T,A}\mu(x,\lambda_1) , \lambda_2 - \delta_{T,A} [\lambda_2 , x ] + D_{T,A}\mu(x,\lambda_2) \rangle + \dots \\
&= \langle \lambda_1 , \lambda_2 \rangle + \langle \lambda_1 , - \delta_{T,A}[\lambda_2 , x] + D_{T,A} \mu(x,\lambda_2) \rangle \\
& \; \; \; \; \; \; + \langle -\delta_{T,A} [\lambda_1 , x] D_{T,A} \mu(x,\lambda_1) , \lambda_2 \rangle + \dots \\
&= \langle \lambda_1 , \lambda_2 \rangle,
\end{align*}
where we have used the fact that $D_{T,A} \lambda_i = D_{T,A}^* \lambda_i = 0$. This shows that $f$ does in fact define normal coordinates at $A$.\\

We now consider the hermitian form $\Omega(\alpha , \beta)$. From (\ref{equomega}) we find
\begin{align*}
\Omega( a(x,\lambda_1) , a(x,\lambda_2) ) &= \frac{1}{2} \int_M B( \lambda_1 - \delta_{T,A}[\lambda_1,x] \wedge \lambda_2 - \delta_{T,A}[\lambda_2,x]) \wedge \omega \wedge \eta + \dots \\
&= \frac{1}{2} \int_M  B(\lambda_1 \wedge \lambda_2) \wedge \omega \wedge \eta
- \int_M B( \lambda_1 \wedge \delta_{T,A}[\lambda_2,x]) \wedge \omega \wedge \eta \\
& \; \; \; \; \; \;  - \int_M B( \delta_{T,A} [\lambda_1,x] \wedge \lambda_2) \wedge \omega \wedge \eta + \dots
\end{align*}
so it will suffice to show that $\int_M B(\delta_{T,A}[\lambda_1, x] \wedge \lambda_2 ) \wedge \omega \wedge \eta = 0$. But this is clear since $\int_M B(\delta_{T,A}[\lambda_1, x] \wedge \lambda_2 ) \wedge \omega \wedge \eta = \langle D_{T,A}^* G_{T,A} [\lambda_1 , x] , J \lambda_2 \rangle = \langle G_{T,A}[\lambda_1,x] , D_{T,A} J \lambda_2 \rangle = 0$, since $J \lambda_2$ is $D_{T,A}$-harmonic. This shows that $\Omega$ is closed. In fact, since the $1$-jets of $h$ and $\Omega$ are constant at $x=0$ in these coordinates, the same is true of $\mathcal{J}$, which implies that $\mathcal{J}$ is integrable and $\mathcal{M}^*$ is K\"ahler.
\end{proof}

Next we consider the case where $M$ has a transverse Calabi-Yau structure. From Corollary \ref{corcy} we see that the moduli space $\mathcal{M}^*$ of irreducible ASD contact instantons on $M$ for a principal $G$-bundle $P$ is smooth. We show that $\mathcal{M}^*$ is naturally a hyperK\"ahler manifold. The transverse K\"ahler identities apply to each of these complex structures showing that the space $\mathcal{H}^1_T$ of harmonic forms is closed under the action of $J_1,J_2,J_3$. This defines almost complex structures $\mathcal{J}_1 = \mathcal{J}, \mathcal{J}_2, \mathcal{J}_3$ on $\mathcal{M}^*$ which are hermitian with respect to the $L^2$-metric $h$. This gives an $Sp(m)$-structure on $\mathcal{M}^*$, where $dim_{\mathbb{R}}(\mathcal{M}^*) = 4m$. Let $\Omega_1,\Omega_2,\Omega_3$ be the associated K\"ahler forms, which are given by expressions of the same form as (\ref{equomega}). By the same argument used in the proof of Proposition \ref{equomega}, the forms $\Omega_i$ are closed and the complex structures $\mathcal{J}_i$ are integrable for $i=1,2,3$, giving:
\begin{proposition}
Let $M$ have a transverse Calabi-Yau structure. Then for any principal $G$-bundle $P$ the moduli space $\mathcal{M}^*$ of irreducible ASD contact instantons is a hyperK\"ahler manifold.
\end{proposition}


\section{Transverse index computations}\label{sectrind}

In this section our aim is to compute the dimension of the moduli space of contact instantons. We have seen that for irreducible contact instantons with vanishing $H^2_B(\mathfrak{g}_P)$ this amounts to computing the index of the basic deformation complex (\ref{basicdc}). This complex is elliptic transverse to the foliation $\mathcal{F}$ of $M$ by the Reeb vector field $\xi$. Determining the index of a complex transverse to a foliation is a notoriously difficult problem, so our first step is to replace the foliation by a group action. This still leaves us with a difficult index problem, but one that is tractable in some cases.\\

We take $M$ to be a compact $K$-contact manifold, $P \to M$ a principal $G$-bundle, where $G$ is compact, connected, semisimple and let $\mathfrak{g}_P$ be the adjoint bundle. Let $\nabla$ be a contact instanton on $P$ which for argument's sake will be anti-self-dual. Recall that as $M$ is $K$-contact we have a torus $T^r$ acting on $M$ by isometries, defined as the closure of the $1$-parameter subgroup $\{ exp(t\xi) \}$ generated by $\xi$. Let $G' = G/Z(G)$ and $P' = P/Z(G)$. Then $\mathfrak{g}_P$ is also the adjoint bundle of $P'$ and $\nabla$ descends to a connection on $P'$. If $\nabla$ is irreducible then by Proposition \ref{proptorus} we have that the action of $T^r$ lifts to an action on $P'$ preserving $\nabla$. From this we have an action of $T^r$ on $\Omega^k_H(M,\mathfrak{g}_P)$ and it is clear that the space of $T^r$-invariant sections of this is precisely $\Omega^k_B(M,\mathfrak{g}_P)$.\\

Consider the two-term complex
\begin{equation}\label{equtrans}
\Omega^1_H(M,\mathfrak{g}_P) \buildrel Q \over \longrightarrow \Omega^0_H(M,\mathfrak{g}_P) \oplus \Omega^+_H(M,\mathfrak{g}_P)
\end{equation}
where $Q = (D_T^* , D_T)$. Then $Q$ is equivariant with respect to the $T^r$-action since $T^r$ preserves the connection $\nabla$. The operator $Q$ in (\ref{equtrans}) is a transversally elliptic operator in the sense of \cite{at} with respect to the action of $T^r$, in fact $Q$ is a transverse Dirac operator. Since $Q$ is not elliptic we can not expect $Ker(Q)$ and $Ker(Q^*)$ to be finite dimensional. However, as shown by Atiyah \cite{at} each irreducible representation of $T^r$ occurs in $Ker(Q)$ and $Ker(Q^*)$ with finite multiplicity and the rate of growth of these multiplicities is such that $ind(Q) = Ker(Q) - Ker(Q^*)$ is well-defined as a distribution on $T^r$. Let $Ker(Q)^{T^r},Ker(Q^*)^{T^r}$ be the subspaces fixed by $T^r$. Then $ind(Q)^{T^r} = dim( Ker(Q)^{T^r} ) - dim( Ker(Q^*)^{T^r} )$ is a well-defined integer.
\begin{proposition}
Let $\nabla$ be an irreducible contact instanton with $H^2_B(\mathfrak{g}_P) = 0$. The dimension of $\mathcal{M}$ around $\nabla$ is given by $dim(H^1_B(\mathfrak{g}_P)) = ind(Q)^{T^r}$.
\end{proposition}
\begin{proof}
Clearly $Ker(Q)^{T^r} = \{ \alpha \in \Omega^1_B(M,\mathfrak{g}_P) \, | \, D_T \alpha = D_T^* \alpha = 0 \} = \mathcal{H}_T^1 \simeq H_B^1(\mathfrak{g}_P)$. In a similar fashion we find $Ker(Q^*)^{T^r} \simeq \mathcal{H}^0_T \oplus \mathcal{H}^2_T = 0$.
\end{proof}
Following \cite{at} we let $T_{T^r}M$ denote the subset of $TM$ consisting of tangent vectors orthogonal to the action of $T^r$. This is typically not a vector bundle, but there is still a natural projection map $\pi \colon  T_{T^r}M \to M$. Transverse ellipticity of $Q$ ensures that the symbol $\sigma_Q(\lambda)$ of $Q$ is an isomorphism for all $\lambda \in T_{T^r}M$ away from the zero section. The symbol complex
\begin{equation}\label{equsymcom}
\pi^*(\mathfrak{g}_P \otimes H^*) \buildrel \sigma_Q(\lambda) \over \longrightarrow \pi^*(\mathfrak{g}_P \otimes (  \mathbb{R} \oplus \wedge^+ H^* ))
\end{equation}
of (\ref{equtrans}) then defines an element in $K_{T^r}( T_{T^r} M )$, the equivariant $K$-theory group of the space $T_{T^r}M$. It is shown in \cite{at} that the index $ind(Q)$ depends only on the $K$-theory class of the symbol complex (\ref{equsymcom}) in $K_{T^r}( T_{T^r} M )$. Moreover the symbol complex is completely determined by the contact structure $H \subset TM$ and the lifted action of $T^r$ on the adjoint bundle $\mathfrak{g}_P$. The problem of computing the dimension of $\mathcal{M}$ at a smooth point has been reduced to the following comparatively simpler problem: given a principal $G$-bundle $P$ and a lift of $T^r$ to $P' = P/Z(G)$, evaluate the index of (\ref{equsymcom}).


\subsection{Quasi-regular case}\label{secqreg}

Suppose that $M$ is quasi-regular. In this case $\xi$ must generate a $1$-dimensional circle action on $M$ and the orbits of $\xi$ are the orbits of the circle action. The quotient $X = M/U(1)$ inherits the structure of a symplectic orbifold with cyclic uniformizing groups and as in Example \ref{exorbi}, $M$ is the total space of an orbifold principal circle bundle over $X$. The quotient map $\pi \colon  M \to X$ is a Seifert fibration. \\

Let $P \to M$ be a principal $G$-bundle on $M$ with a lift of the $T^r = U(1)$-action on $M$ to the adjoint bundle $\mathfrak{g}_P$. Then the symbol complex (\ref{equsymcom}) is defined and has a $U(1)$-equivariant index $ind(P)$. Let $ind(P)^{U(1)} \in \mathbb{Z}$ denote the $U(1)$-invariant part of $ind(P)$. Note that $\mathfrak{g}_P$ is an orbifold vector bundle on $X$ and that adjoint-valued basic differential forms on $M$ correspond to adjoint-valued differential forms on $X$. This gives an identification $\Omega^*_B(M,\mathfrak{g}_P) = \Omega^*(X,\mathfrak{g}_P)$ under which the basic deformation complex on $M$ becomes
\begin{equation}\label{equorbcpx}
0 \to \Omega^0(X,\mathfrak{g}_P) \buildrel D \over \longrightarrow \Omega^1(X,\mathfrak{g}_P) \buildrel D \over \longrightarrow \Omega^+(X,\mathfrak{g}_P) \to 0.
\end{equation}
This is an elliptic complex on $X$ and it is clear that $ind(P)^{U(1)}$ is the index of (\ref{equorbcpx}). The index $ind(P)^{U(1)}$ may therefore be computed by orbifold index theory. To simplify matters we will restrict to the special case where $M$ is Sasakian and $X$ has isolated singularities. Assume that $P$ admits an ASD contact instanton $A$. Then since $A$ is a transverse connection it can be regarded as a connection on the orbifold bundle $\mathfrak{g}_P$. Since $M$ is Sasakian, $X$ is a complex orbifold and the curvature of $A$ is of type $(1,1)$. This allows us to define a $\overline{\partial}$-operator $\overline{\partial}_A$ on $(\mathfrak{g}_P)_{\mathbb{C}}$ and we have the Dolbeault complex
\begin{equation}\label{equdolb}
0 \to \Omega^{0,0}(X,\mathfrak{g}_P) \buildrel \overline{\partial}_A \over \longrightarrow \Omega^{0,1}(X,\mathfrak{g}_P) \buildrel \overline{\partial}_A \over \longrightarrow \Omega^{0,2}(X,\mathfrak{g}_P) \to 0.
\end{equation}
Let $ind(\overline{\partial}_A) = \sum_i (-1)^i dim_{\mathbb{C}} (H^i_{\overline{\partial}_A})$ be the index of this complex. Then from Proposition \ref{prophodge} we have $ind(P)^{U(1)} = 2 \, ind(\overline{\partial}_A)$, so it remains to determine $ind(\overline{\partial}_A)$.\\

Let $x_1,\dots , x_k$ be the singular points of $X$ with uniformizing groups $\Gamma_{x_j} = \mathbb{Z}_{m_j}$ cyclic of order $m_j$. Around $x_j$ we can find an orbifold chart of the form $U_j / \mathbb{Z}_{m_j}$, where $U_j$ is a connected open subset of $\mathbb{C}^2$ containing the origin. For $k \in \mathbb{Z}_{m_j}$ we let $g_k \colon  \mathbb{C}^2 \to \mathbb{C}^2$ be the corresponding linear transformation of $\mathbb{C}^2$. We may assume this action of $\mathbb{Z}_{m_j}$ on $\mathbb{C}^2$ is of the form $g_k(z_1,z_2) = (\zeta^k z_1 , \zeta^{w_j k} z_2)$, where $\zeta_j = {\rm exp}(2\pi i/m_j)$, $0 < w_j < m_j$ is coprime to $m_j$ and $k \in \mathbb{Z}_{m_j}$. We have a real representation of $\mathbb{Z}_{m_j}$ on the fibre of $\mathfrak{g}_P$ over $x_j$ and we let $\chi_{\mathfrak{g}_P,j} \colon  \mathbb{Z}_{m_j} \to \mathbb{R}$ denote the character of this representation.
\begin{proposition}
The index $ind(\overline{\partial}_A)$ of the Dolbeault complex (\ref{equdolb}) is given by:
\begin{equation}\label{equindex1}
ind(\overline{\partial}_A) = p_{1,B}( \mathfrak{g}_P )[X] + \frac{1}{2}dim(G)( 1 - b^1_B(M) + b^+_B(M) ) + \sum_j \frac{1}{m_j} \sum_{k=1}^{m_j-1} \frac{ \chi_{\mathfrak{g}_P,j}(k) - dim(G)  }{ det(1 - g_k |_{T^*_{x_j} X } ) }
\end{equation}
where $b^1_B(M) = dim( H^1_B(M) )$, $b^+_B(M) = dim( H^+_B(M) )$ and $p_{1,B}( \mathfrak{g}_P )[X]$ is the basic Pontryagin class of $\mathfrak{g}_P$ integrated over $X$.
\end{proposition}
\begin{remark}
We have assumed $P$ admits a contact instanton in order to define the Dolbeault complex, however this assumption is not necessary for the index computation. In general one arrives at the same result $ind(P)^{U(1)} = 2 \, ind(\overline{\partial}_A)$, where $ind(\overline{\partial}_A)$ is given by (\ref{equindex1}).
\end{remark}
\begin{proof}
Since $X$ has only isolated orbifold singularities the orbifold Riemann-Roch theorem gives \cite{kaw2}
\begin{equation*}
ind(\overline{\partial}_A) = \int_X Ch((\mathfrak{g}_P)_{\mathbb{C}}) Td(X) + \sum_j \frac{1}{m_j} \sum_{k=1}^{m_j-1} \frac{ \chi_{\mathfrak{g}_P,j}(k)  }{ det(1 - g_k |_{T^*_{x_j} X } ) }.
\end{equation*}
The degree $4$ component of the integrand is $p_{1,B}(\mathfrak{g}_P) + dim(G)Td(X)$. A second application of orbifold Riemann-Roch gives
\begin{equation*}
h^{0,0}(X) - h^{0,1}(X) + h^{0,2}(X) = \int_X Td(X) + \sum_j \frac{1}{m_j} \sum_{k=1}^{m_j-1} \frac{1}{ det(1 - g_k |_{T^*_{x_j} X } ) },
\end{equation*}
where $h^{p,q}(X)$ denotes the orbifold Hodge numbers of $X$. Combining these and using $h^{0,0}(X) - h^{0,1}(X) + h^{0,2}(X) = \frac{1}{2}(1 -b^1_B(M) + b^+_B(M))$ we arrive at (\ref{equindex1}).
\end{proof}
Suppose that the action of $\mathbb{Z}_{m_j}$ on $(\mathfrak{g}_P)_{x_j}$ has weights $u_l$, for $l = 1, \dots , dim(G)$. Then $\chi_{\mathfrak{g}_P,j}(k) - dim(G) = \sum_l (\zeta^{ku_l} - 1) = -2 \sum_l \sin^2( \frac{\pi k u_l }{m_j} )$ and we may rewrite (\ref{equindex1}) as
\begin{equation*}
\begin{aligned}
ind(\overline{\partial}_A) &= p_{1,B}( \mathfrak{g}_P )[X] + \frac{1}{2}dim(G)( 1 - b^1_B(M) + b^+_B(M) ) \\
& \; \; \; \; -\frac{1}{2} \sum_j \frac{1}{m_j} \sum_{k=1}^{m_j-1} \sum_{l=1}^{dim(G)} \sin^2 \left( \frac{\pi k u_l }{m_j} \right)\left( 1 -  \cot \left( \frac{\pi k}{m_j} \right) \cot \left( \frac{\pi k w_j}{m_j} \right) \right).
\end{aligned}
\end{equation*}
In order to compute the index one still needs to integrate $p_{1,B}( \mathfrak{g}_P )$ over $X$. Let us see how this can be done in the special case where $P$ is the principal $SO(3)$-bundle associated to $\wedge^- H^*$, for which $\mathfrak{g}_P = \wedge^- H^*$.
\begin{proposition}\label{propindex2}
When $\mathfrak{g}_P = \wedge^- H^*$, we have
\begin{equation}\label{equindex2}
ind(\overline{\partial}_A) = \frac{5}{4}( 3\tau_B(M) - \chi_B(M) ) + \sum_j \left( 2 - \frac{w_j + w'_j}{m_j} + 12 s(w_j ; m_j) \right),
\end{equation}
where $w'_j$ is the unique integer $0 < w'_j < m_j$ with $w_j w'_j = 1 ({\rm mod} \; m_j)$, $\tau_B(M)$ is the basic signature of $M$, $\chi_B(M)$ the basic Euler characteristic and $s(w_j ; m_j)$ is the Dedekind sum \cite{ap}
\begin{equation*}
s(w_j;m_j) = \frac{1}{4m_j} \sum_{k=1}^{m_j-1} \cot \left( \frac{\pi k}{m_j} \right) \cot \left( \frac{\pi k w_j}{m_j} \right).
\end{equation*}
\end{proposition}
\begin{proof}
For $\mathfrak{g}_P = \wedge^- H^*$ we have $\chi_{\mathfrak{g}_P,j}(k) = 1 + \zeta^{k(w_j-1)} + \zeta^{-k(w_j-1)}$. Then
\begin{equation*}
\begin{aligned}
\sum_{k=1}^{m_j-1} \frac{ \chi_{\mathfrak{g}_P,j}(k) - dim(G) }{ det(1 - g_k |_{T^*_{x_j} X } ) } &= \sum_{k=1}^{m_j-1} \frac{ \zeta^{k(w_j-1)} + \zeta^{-k(w_j-1)} - 2}{(1-\zeta^{-k})(1-\zeta^{-kw_j})} \\
&= \sum_{k=1}^{m_j-1} \frac{ \zeta^{2w_jk} + \zeta^{2k} - 2}{(1 - \zeta^k)(1-\zeta^{kw_j})} \\
&= - \sum_{k=1}^{m_j-1} \left( \frac{1+\zeta^{kw_j}}{1-\zeta^k} + \frac{1+\zeta^k}{1-\zeta^{kw_j}} \right).
\end{aligned}
\end{equation*}
An elementary computation shows that
\begin{align*}
\sum_{k=1}^{m_j-1} \frac{1+\zeta^{kw_j}}{1-\zeta^k} &= w_j-1, & \sum_{k=1}^{m_j-1} \frac{1+\zeta^{k}}{1-\zeta^{kw_j}} &= w'_j - 1.
\end{align*}
Substituting into (\ref{equindex1}) we have
\begin{equation}\label{equindexpart}
ind(\overline{\partial}_A) = p_{1,B}( \wedge^- H^* )[X] +  \frac{3}{2}( 1 - b^1_B(M) + b^+_B(M) ) + \sum_j \frac{2 - w_j - w'_j}{m_j}.
\end{equation}
Moreover the basic Pontryagin class of $\wedge^- H^*$ is given by
\begin{equation*}
p_{1,B}(\wedge^- H^*) = p_{1,B}( H ) - 2 e_B( H ),
\end{equation*}
where $p_{1,B}(H)$ is the basic Pontryagin class and $e_B(H)$ the basic Euler class of $H$. By the orbifold signature theorem \cite{kaw1} we have
\begin{equation*}
\begin{aligned}
\tau_B(M) &= \frac{1}{3} p_{1,B}(H)[X] + \sum_j \frac{1}{m_j} \sum_{k=1}^{m_j-1} \frac{( 1 + \zeta^{-k})(1+\zeta^{-kw_j})}{(1-\zeta^{-k})(1-\zeta^{-kw_j})} \\
&= \frac{1}{3} p_{1,B}(H)[X] - \sum_j 4 s(w_j ; m_j).
\end{aligned}
\end{equation*}
and by the orbifold Gauss-Bonnet theorem \cite{kaw2}
\begin{equation*}
\chi_B(M) = e_B(H)[X] + \sum_j \left( 1 - \frac{1}{m_j} \right).
\end{equation*}
Substituting into (\ref{equindexpart}) and using $(1 - b^1_B(M) + b^+_B(M) ) = \frac{1}{2}( \chi_B(M) + \tau_B(M))$ we obtain (\ref{equindex2}).
\end{proof}
\begin{remark}
The Dedekind sums $s(w;m)$ have the property that $6ms(w;m) \in \mathbb{Z}$ and $12wm s(w;m) = w^2 + 1 ({\rm mod} \; m)$ \cite{ap}. This implies that each summand $2 - \frac{w_j + w'_j}{m_j} + 12 s(w_j ; m_j)$ in (\ref{equindex2}) is an integer.  
\end{remark}

\begin{example}
Consider the case where $M$ is a compact Sasaki $5$-manifold with transverse Calabi-Yau structure. As in Remark \ref{remtranscy}, we have that $M$ is quasi-regular with leaf space $X$ a Calabi-Yau orbifold. Furthermore it can be shown that $X$ has no branch divisor \cite{bg}, meaning that it only has orbifold singularities of codimension $2$. Since $X$ has complex dimension $2$ this means that the orbifold singularities are isolated and since $X$ is Calabi-Yau, the action of the uniformizing groups $\mathbb{Z}_{m_j}$ must be of the form $g_k(z_1,z_2) = (\zeta^k z_1 , \zeta^{- k} z_2)$. This is a du Val singularity of type $A_{m_j-1}$ and we note that $w_j = w'_j = m_j-1$. Suppose also that $M$ is simply-connected, then $\tau_B(M) = 4 - h^{1,1}(X)$ and $\chi_B(M) = 4 + h^{1,1}(X)$. We remark also that in this case one can show that $M$ is diffeomorphic to a connected sum of $b^2(M) = h^{1,1}(X)+1$ copies of $S^2 \times S^3$ \cite{bg}. Examples of such $M$ are given by taking $X$ to be an orbifold $K3$-surface, in particular such orbifolds were constructed by Reid from hypersurfaces in weighted projective spaces \cite{rei} (see also \cite{bg}). Such an $M$ has a transverse Einstein structure, hence the transverse Levi-Civita connection on $\wedge^- H^*$ is anti-self-dual.\\

From Proposition \ref{propindex2} we obtain the index
\begin{equation}\label{equpartindex}
-ind(\overline{\partial}_A) = 5 h^{1,1}(X) - 10 + \sum_j (m_j-3).
\end{equation}
We can simplify this further by noting that since $X$ is Calabi-Yau, the Levi-Civita connection on $\wedge^+ H^*$ is flat giving $p_{1,B}(H)[X] + 2e_B(H)[X] = 0$. Using the signature and Gauss-Bonnet theorems as in the proof of Proposition \ref{propindex2} gives $20 - h^{1,1} = \sum_j (m_j - 1)$. Combining with (\ref{equpartindex}) gives:
\begin{equation*}
-ind(\overline{\partial}_A) = 90 - 2 \sum_j (2m_j-1).
\end{equation*}
Next we seek conditions under which the Levi-Civita connection on $\wedge^- T^*X$ is irreducible. Suppose to the contrary that it is reducible. Then the Levi-Civita connection on $X$ reduces to a $U(1)$-connection. This means that $R^-_-$, the $\wedge^- T^*X \otimes \wedge^- T^*X$-part of the curvature has rank $\le 1$. Moreover the trace of $R^-_-$ is $1/4$ times the scalar curvature, which is zero since $X$ is Calabi-Yau. So in fact the Levi-Civita connection is flat on $\wedge^- T^*X$ and it follows that $X$ is flat. In this case Gauss-Bonnet gives $\chi_B(M) = 4 + h^{1,1}(X) = \sum_j (m_j-1)/m_j$. However we have already shown that $20-h^{1,1}(X) = \sum_j (m_j-1)$. Combining these gives
\begin{equation}\label{equflatcond}
\sum_j \frac{m_j^2 - 1}{m_j} = 24.
\end{equation}
If (\ref{equflatcond}) does not hold then the Levi-Civita connection $\nabla$ on $\wedge^- H^*$ is irreducible. In such a case $\nabla$ belongs to a smooth moduli space of contact instantons forming a hyperK\"ahler manifold of complex dimension equal to $-ind(\overline{\partial}_A) = 90 - 2 \sum_j (2m_j-1)$. To give just one specific example from \cite{rei}, a generic hypersurface $X$ of degree $42$ in the weighted projective space $\mathbb{CP}(1,6,14,21)$ gives an orbifold $K3$-surface with three singularities of types $A_1,A_2,A_6$, giving $-ind(\overline{\partial}_A) = 90 - 2(3+5+13) = 48$. Since (\ref{equflatcond}) is not satisfied this gives a hyperK\"ahler moduli space of complex dimension $48$.
\end{example}


\subsection{$Y^{p,q}$ spaces}\label{secypq}

For irregular $K$-contact manifolds, the computation of the transverse index is generally difficult. One case where the index becomes computable is when $M$ is a toric Sasaki $5$-manifold. For sake of definiteness we will carry out the index computation for a class of irregular toric Sasaki-Einstein $5$-metrics on $S^2 \times S^3$, the $Y^{p,q}$ spaces. The method of computation could equally be applied to any compact toric Sasaki $5$-manifold.\\

For each pair of coprime integers $p,q$ with $p > q$ the space $Y^{p,q}$ is a toric Sasaki-Einstein metric on $S^2 \times S^3$ \cite{gmsw}. The Sasakian structure on $Y^{p,q}$ is quasi-regular whenever $4p^2 - 3q^2$ is a perfect square, otherwise $Y^{p,q}$ is irregular of rank $2$. Since $Y^{p,q}$ is Sasaki-Einstein the transverse Levi-Civita connection $\overline{\nabla}$ on $\wedge^{-} H^*$ is an ASD contact instanton. We will show that $\overline{\nabla}$ is irreducible, except in the case $(p,q) = (1,0)$, and we compute the dimension of the moduli space of contact instantons around $\overline{\nabla}$ in the irregular case.

\begin{proposition}\label{propypqirred}
Let $M = Y^{p,q}$. The transverse Levi-Civita connection on $\wedge^- H^*$ is irreducible provided $(p,q) \neq (1,0)$.
\end{proposition}
\begin{proof}
Let $R_T$ denote the curvature of $\overline{\nabla}$ and $(R_T)^-_-$ the $\wedge^- H^* \otimes \wedge^- H^*$-component of $R_T$. Viewing $(R_T)^-_-$ as a map $(R_T)^-_- \colon  \wedge^- H^* \to \wedge^- H^*$ it will suffice to show that $(R_T)^-_-$ does not have rank $\le 1$ everywhere. From \cite{gmsw} there is an open subset of $M$ and local coordinates on the leaf space for which the transverse metric $g_T$ has the form
\begin{equation*}
g_T = \frac{1}{\Delta} d\rho^2 + \frac{\rho^2}{4}( \tilde{\sigma}_1^2 + \tilde{\sigma}_2^2 + \Delta \tilde{\sigma}_3^2 ),
\end{equation*}
provided $(p,q) \neq (1,0)$. Here $\rho$ is a local coordinate, $\tilde{\sigma}_i$ are the left-invariant $1$-forms on $SU(2)$, $\Delta$ is given by
\begin{equation*}
\Delta = 1 + \frac{4(a-1)}{27} \frac{1}{\rho^4} - \rho^2
\end{equation*}
and $a \in (0,1)$ is a constant (see \cite{gmsw} for further details). Define an orthonormal coframe $e^1, \dots , e^4$ as follows:
\begin{equation*}
\begin{aligned}
e^1 &= \frac{\rho}{2}\tilde{\sigma}_1, & e^2 &= \frac{\rho}{2}\tilde{\sigma}_2, & e^3 &= \frac{\rho \sqrt{\Delta}}{2} \tilde{\sigma}_3, & e^4 &= \frac{d\rho}{\sqrt{\Delta}}.
\end{aligned}
\end{equation*}
By computing the curvature of the transverse metric one finds
\begin{equation*}
\begin{aligned}
(R_T)( e^{12} - e^{34}) &= \left( \frac{8-8\Delta}{\rho^2} - 6 \right)(e^{12} - e^{34}) \\
(R_T)( e^{13} + e^{24}) &= \left( \frac{4\Delta-4}{\rho^2} + 6 \right)(e^{13} + e^{24}) \\
(R_T)(e^{14} - e^{23}) &= \left( \frac{4\Delta-4}{\rho^2} + 6 \right)(e^{14} - e^{23}).
\end{aligned}
\end{equation*}
Then since $\frac{4\Delta-4}{\rho^2} + 6$ can not be identically zero we find that $(R_T)^-_-$ does not always have rank $\le 1$.
\end{proof}

Assume that $(p,q) \neq (1,0)$ and let $\mathcal{M}^*(\wedge^- H^*)$ denote the moduli space of irreducible ASD contact instantons on $\wedge^- H^*$. From Proposition \ref{propypqirred} we have that $\mathcal{M}^*(\wedge^- H^*)$ is a non-empty K\"ahler manifold. Let $\mathcal{M}^*(\wedge^- H^*)_{\overline{\nabla}}$ denote the connected component of $\mathcal{M}^*(\wedge^- H^*)$ containing $\overline{\nabla}$. 
\begin{proposition}
Let $(p,q)$ be such that $Y^{p,q}$ is irregular, that is $4p^2 - 3q^2$ is not a square. Then $\mathcal{M}^*(\wedge^- H^*)_{\overline{\nabla}}$ is a K\"ahler manifold of complex dimension $3$.
\end{proposition}
\begin{proof}
Our proof follows closely the index computation given in Appendix D of \cite{qz} and we will make use of their notation. Let $T^3$ be the $3$-torus acting on $Y^{p,q}$ and let $s,t,u$ denote coordinates on $T^3$. The vector fields generating $T^3$ which correspond to these coordinates are denoted $e_1,e_3,\alpha$ in \cite{qz}. Recall that $Y^{p,q}$ is of Reeb type, meaning $\xi$ is a vector field generated by the $T^3$-action. Moreover since $Y^{p,q}$ is irregular of rank $2$, the closure of the subgroup of $T^3$ generated by $\xi$ is a $2$-torus $T^2_\xi \subset T^3$. From \cite{qz} we find $T^2_\xi$ is the subgroup generated by $e_1,\alpha$. The coordinate $t$ identifies the quotient $T^3/T^2_\xi$ with a $1$-torus $T^1$.\\

Consider the transverse Dolbeault complex:
\begin{equation}\label{equypqdolb}
0 \to \Omega^{0,0}_H(M,\wedge^- H^*_{\mathbb{C}}) \buildrel \overline{\partial}_A \over \longrightarrow \Omega^{0,1}_H(M,\wedge^- H^*_{\mathbb{C}}) \buildrel \overline{\partial}_A \over \longrightarrow \Omega^{0,2}_H(M,\wedge^- H^*_{\mathbb{C}}) \to 0,
\end{equation}
where $H_\mathbb{C} = H \otimes \mathbb{C}$ and we have used that the adjoint bundle associated to $\wedge^- H^*$ is again $\wedge^- H^*$. The group $T^3$ lifts to an action on this complex making (\ref{equypqdolb}) a transverse elliptic complex. Let $ind(\overline{\partial}_A)$ denote the index of this complex, which as we recall is a distribution on $T^3$. Let $ind(\overline{\partial}_A)^{T^2_\xi}$ denote the $T^2_\xi$-invariant part of $ind(\overline{\partial}_A)$. Note that this is a distribution on $T^3$ which is invariant under $T^2_\xi$, so we may regard it as a distribution on the quotient $T^1$. In fact, by imposing invariance under $T^2_\xi$, we are passing to the basic Dolbeault complex, hence as representations of $T^1$ we have
\begin{equation*}
ind(\overline{\partial}_A)^{T^2_\xi} = H^0_{\overline{\partial}_A}(\wedge^- H^*_{\mathbb{C}}) - H^1_{\overline{\partial}_A}(\wedge^- H^*_{\mathbb{C}}) + H^2_{\overline{\partial}_A}(\wedge^- H^*_{\mathbb{C}}) = -H^1_{\overline{\partial}_A}(\wedge^- H^*_{\mathbb{C}}), 
\end{equation*}
where $H^k_{\overline{\partial}}(\wedge^- H^*_{\mathbb{C}})$ is the degree $k$ basic Dolbeault cohomology of $\wedge^- H^*_{\mathbb{C}}$. To get the complex dimension of $\mathcal{M}^*(\wedge^- H^*)_{\overline{\nabla}}$, we simply need to evaluate $-ind(\overline{\partial}_A)^{T^2_\xi}$ at $t=1$. Note that since $H^1_{\overline{\partial}_A}(\wedge^- H^*_{\mathbb{C}})$ is finite dimensional, this evaluation is well-defined.\\

The index of the transverse Dolbeault complex is computed in \cite{qz} for the case of a trivial bundle. We will simply adapt this calculation to the case of $\wedge^- H^*_{\mathbb{C}}$. The idea is to perturb the symbol complex along a vector field in the $T^3$-action so that the complex is an isomorphism away from the closed Reeb orbits. In this way the index computation localises to a sum of contributions over these orbits. Consider a closed orbit $\mathcal{O} \subset Y^{p,q}$ and note that $\mathcal{O}$ is an embedded circle. For such an orbit the torus $T^3$ can be decomposed into a product $T^3 = T^2 \times T^1$, where the $T^1$ acts as translation along the orbit and the $T^2$ subgroup acts on the normal bundle to $\mathcal{O}$. The $T^2$ subgroup also acts on the fibres of $\wedge^- H^*_{\mathbb{C}}|_{\mathcal{O}}$ according to some representation and we let $\chi( \wedge^- H^*_{\mathbb{C}}|_{\mathcal{O}})$ denote the character of this representation. On pulling back by the projection $T^3 \to T^2$ we will regard $\chi( \wedge^- H^*_{\mathbb{C}}|_{\mathcal{O}})$ as a character of $T^3$. \\

We will determine the characters $\chi( \wedge^- H^*_{\mathbb{C}}|_{\mathcal{O}})$. For this recall that the $Y^{p,q}$ spaces are given by a Delzant-type construction starting from a moment cone \cite{ms}. Using this construction it is possible to determine the weights of the $T^2$ subgroups on the normal bundle to each closed orbit $\mathcal{O}$. These weights can be read off Table (37) in \cite{qz}. Now since $\mathcal{O}$ is a Reeb orbit there is an isomorphism between the normal bundle to $\mathcal{O}$ and the restriction $H|_{\mathcal{O}}$, thus the weights of the action on the normal bundle determines the character $\chi( \wedge^- H^*_{\mathbb{C}}|_{\mathcal{O}})$. We let $\{ \mathcal{O}_{ij} \}_{0 \le i,j \le 1}$ denote the closed orbits. These correspond to the coordinate patches $U_{ij}$ in \cite{qz}. To adapt the index calculation in \cite{qz} to the case of $\wedge^- H^*_{\mathbb{C}}$ one simply has to insert the character $\chi( \wedge^- H^*_{\mathbb{C}}|_{\mathcal{O}_{ij}})$ into the index contribution over $\mathcal{O}_{ij}$. Thus the contributions to the index (using the notation of \cite{qz}) are as follows:
\begin{equation*}
\begin{aligned}
{\rm orbit \;} \mathcal{O}_{00} &: \; (1+st^{-1} + ts^{-1})\left[ \frac{1}{1-s^{-1}} \right]^+ \left[ \frac{1}{1-t^{-1}} \right]^+ \delta(1-u) \\
{\rm orbit \;} \mathcal{O}_{01} &: \; (1+st^3+s^{-1}t^{-3})\left[ \frac{1}{1-(st^2)^{-1}} \right]^+ \left[ \frac{1}{1-t} \right]^+ \delta(1-ut^{q-p}) \\
{\rm orbit \;} \mathcal{O}_{10} &: \; (1+st^{-1} + ts^{-1})\left[ \frac{1}{1-s^{-1}} \right]^- \left[ \frac{1}{1-t^{-1}} \right]^+ \delta(1-us^p) \\
{\rm orbit \;} \mathcal{O}_{11} &: \; (1+st^3+s^{-1}t^{-3})\left[ \frac{1}{1-(st^2)^{-1}} \right]^- \left[ \frac{1}{1-t} \right]^+ \delta(1-us^pt^{p+q}).
\end{aligned}
\end{equation*}
The index is the sum of these four contributions. Next to take the $T^2_\xi$-invariant part $ind(\overline{\partial}_A)^{T^2_\xi}$ of the index, one needs to extract the terms which are independent of $s$ and $u$, giving
\begin{equation*}
-ind(\overline{\partial}_A)^{T^2_\xi} = t^{-1} + 1 + t.
\end{equation*}
Finally to get the dimension of the moduli space, we set $t=1$ giving $dim(H^1_{\overline{\partial}_A}(\wedge^- H^*_{\mathbb{C}})) \linebreak = 3$.
\end{proof}


\bibliographystyle{amsplain}

\end{document}